\documentclass{IEEEtran}
\usepackage{amssymb,amsmath,amsthm,graphicx,subfigure,url}
\usepackage{tikz}
\usetikzlibrary{arrows.meta}
\usepackage{siunitx}
\usepackage{cleveref}
\usepackage{diagbox,multirow}

\newcommand{\norm}[1]{\ensuremath{\left\| #1 \right\|}}
\newcommand{\bracket}[1]{\ensuremath{\left[ #1 \right]}}
\newcommand{\braces}[1]{\ensuremath{\left\{ #1 \right\}}}

\newcommand{\tr}[1]{\mathrm{tr}\ensuremath{\negthickspace\bracket{#1}}}
\newcommand{\trs}[1]{\mathrm{tr}\ensuremath{[#1]}}

\newcommand{\deriv}[2]{\ensuremath{\frac{\partial #1}{\partial #2}}}

\newcommand{\SO}{\ensuremath{\mathsf{SO(3)}}}
\newcommand{\T}{\ensuremath{\mathsf{T}}}

\newcommand{\so}{\ensuremath{\mathfrak{so}(3)}}

\renewcommand{\Re}{\ensuremath{\mathbb{R}}}

\newcommand{\Sph}{\ensuremath{\mathsf{S}}}

\newcommand{\Q}{\ensuremath{\mathsf{Q}}}

\newcommand{\dexp}{\mathrm{dexp}}
\newcommand{\E}{\mathrm{E}}
\newcommand{\ad}{\ensuremath{\mathrm{ad}}}

\DeclareMathOperator*{\argmax}{arg\,max}
\DeclareMathOperator*{\argmin}{arg\,min}

\newcommand{\fm}{\ensuremath{\mathcal{E}}}

\date{}

\newtheorem{definition}{Definition}[section]
\newtheorem{lem}{Lemma}[section]
\newtheorem{prop}{Proposition}[section]

\newtheorem{theorem}{Theorem}[section]

\graphicspath{{./Figs/}}

\title{On the Observability of Attitude with Single Direction Measurements}

\author{Weixin Wang, Kanishke Gamagedara, and Taeyoung Lee
\thanks{W. Wang, K. Gamagedara, and T. Lee are with the Department of Mechanical and Aerospace Engineering,  The George Washington University, Washington DC, USA.
{\tt\small  \{wwang442,kanishkegb,tylee\}@gwu.edu}}%
\thanks{This research has been supported in part by NSF under the grant CNS-1837382, and by AFOSR under the grant FA9550-18-1-0288.}}

\begin{document}
\allowdisplaybreaks

\maketitle

\begin{abstract}
The attitude of a rigid body evolves on the three-dimensional special orthogonal group,
and it is often estimated by measuring reference directions, such as gravity or magnetic field, using an onboard sensor.
As a single direction measurement provides a two-dimensional constraint, it has been widely accepted that at least two non-parallel reference directions should be measured, or the reference direction should change over time, to determine the attitude completely.
This paper uncovers an intriguing fact that the attitude can actually be estimated by using multiple measurements of a single, fixed reference direction, provided that the angular velocity and the direction measurements are resolved in appropriate frames, respectively. 
More specifically, after recognizing that the attitude uncertainties propagated over the left-trivialized stochastic kinematics are distinct from those over the right-trivialized one, stochastic attitude observability with single direction measurements is formulated by an information theoretic analysis. 
These are further illustrated by numerical simulations and experiments. 
\end{abstract}

\section{Introduction}

The attitude of a rigid body is the orientation of its body-fixed frame relevant to another reference frame. 
It is defined by the coordinates of each body-fixed axis resolved in the reference frame, leading to the three-by-three orthogonal matrix with determinant one, namely a rotation matrix in the three-dimensional special orthogonal group, $\SO$.
The attitude of a rigid body is often determined by measuring a set of directions that are known in the reference frame, which is referred to as the inertial frame throughout the remainder of this paper. 
For example, an accelerometer attached to an aerial vehicle measures the direction of gravity, and a star tracker on a satellite provides the direction to distant stars. 
These measurements are resolved in the body-fixed frame, which are related to the given coordinates of the reference directions in the inertial frame through the rotation matrix. 
As such, the attitude can be determined by comparing the measured coordinates in the body-fixed frame against the known coordinates in the inertial frame. 
In particular, it is formulated as an optimization problem to minimize the weighted sum of residual errors~\cite{WahSR65}.
This is referred to as Wahba's problem and it has been addressed by various methods~\cite{keat1977analysis,ShuOhJGCD81,MarJAS88}.

The attitude of a rigid body has three degrees of freedom, but a single vector measurement provides a two-dimensional constraint to the attitude.
Therefore, it is well known that at least two direction measurements are required to determine the attitude completely. 
More specifically, if only one reference direction is measured, then any rotation about this vector does not change the measurement, which means the one-dimensional rotation about the reference vector cannot be determined from the measurement. 

However, this does not consider the time evolution of the attitude described by the attitude kinematics. 
Therefore, it remains to check if the attitude can be completely estimated when the fixed reference direction is repeatedly measured together with the angular velocity, after extending the attitude estimation problem into a state estimation of a dynamical system. 
A wide variety of attitude estimators have been developed utilizing unit quaternions, and in particular, the multiplicative extended Kalman filter (MEKF) \cite{LefMarJGCD82} has been successfully applied to various space missions. 
In \cite{barrau2016invariant}, invariant extended Kalman filters are proposed for general Lie groups, and its stability is analyzed. 
Further, a Bayesian attitude estimator is proposed in \cite{LeeITAC18}, which is based on the matrix Fisher distribution defined directly on the special orthogonal group.
In the deterministic sense, a complementary attitude observer has been developed on the special orthogonal group~\cite{MahHamITAC08}.

However, even in these attitude estimators incorporating the attitude kinematics, it has been widely accepted that the rotation about the single reference vector still remains unobservable.
An alternative way to estimate the complete attitude with a single direction measurement is assuming that the reference direction is time-varying in the inertial frame~\cite{batista2012ges,grip2011attitude,LeeLeoPACC07}.
It is also proposed that the complete attitude can be estimated by utilizing a gyroscope that is accurate enough to capture the angular velocity of the Earth~\cite{reis2018nonlinear}.

This paper presents a new perspective to analyze the attitude observability, through which we have discovered two additional cases where the attitude can be estimated completely using a single fixed reference direction.
This is based on a careful analysis regarding how the attitude uncertainties are propagated through the stochastic differential equation on the special orthogonal group representing the attitude kinematics. 
In particular, it is shown that using the left-trivialization to transfer the tangent space of a Lie group to its Lie algebra is not equivalent to the right-trivialization in stochastic settings, in contrast to deterministic settings. 
More specifically for attitude uncertainty propagation, the uncertainty propagated by the angular velocity resolved in the body-fixed frame is distinct from the one propagated by the angular velocity resolved in the inertial frame. In the former, the direction of one-dimensional ambiguity caused by a single direction measurement remains fixed in the inertial frame, and for the latter, it is fixed in the body-fixed frame instead.

\begin{table}
	\caption{\label{table:observability} Attitude observability with a single vector measurement}
	\centering
	\begin{tabular}{l|cc}
		\diagbox[width=10em]{ref. vec.}{ang. vel.} & body-fixed frame & inertial frame \\ \hline
		body-fixed frame & fully observable & not fully observable \\
		inertial frame & not fully observable & fully observable
	\end{tabular}
\end{table}

This explains the fundamental reason why the attitude is not observable through a single inertial reference direction and the angular velocity measured by a gyroscope: the direction of ambiguity caused by the inertial reference direction measurement remains unchanged by the angular velocity resolved in the body-fixed frame.
This leads to two strategies to achieve the complete attitude observability with single direction measurements, namely utilizing the angular velocity resolved in the inertial frame such that the direction of ambiguity is rotated over propagation, or measuring a reference direction fixed to the body such that the next measurement can resolve the ambiguity.
These results are summarized in \Cref{table:observability}.

This observation is more formally studied by introducing the stochastic attitude observability. 
Because the observability of attitude critically depends on how uncertainties propagate, we study the attitude observability in stochastic estimators, instead of the convergence problem of deterministic observers.
Observability in stochastic systems is distinct from deterministic systems, since the state can never be exactly recovered from the measurement due to the presence of noise.
In stochastic estimators, the information carried by the measurement is incorporated into estimates by calculating the distribution of the state conditioned on the measurements.
Therefore one of the criteria for stochastic observability is whether the state is dependent of the measurement, which is quantified by their mutual information~\cite{mohler1988nonlinear,liu2011stochastic}.
Another similar notion of stochastic observability is given by the Fisher information of the state, whose inverse yields a lower bound for the variance of all unbiased estimators, known as the Cram\'{e}r--Rao bound.
The observability can be defined by the condition that the Fisher information matrix is positive-definite~\cite{mohler1988nonlinear}.

In this paper, we will use the positive definiteness of the Fisher information matrix to indicate whether the full attitude is observable.
However, the state considered here is the attitude of a rigid body, which resides in a non-Euclidean manifold.
Thus the usual definition of Fisher information for Euclidean random vectors cannot be directly applied.
The Fisher information of random elements on Lie groups or arbitrary manifolds, and its associated Cram\'{e}r--Rao inequality have been studied in \cite{smith2005covariance,Chi12,hendriks1991cramer} with varying levels of generality,
which have been adopted in a stochastic attitude filter with two vector measurements~\cite{bonnabel2015intrinsic}.
We will use the development in~\cite{smith2005covariance} to calculate the Fisher information for attitude in this paper, 

Moreover, for a given probability density function on $\SO$, the attitude that minimizes the mean squared error may not be unique~\cite{moakher2002means,pennec2006intrinsic}.
This gives an alternative characterization of the attitude observability, since if multiple attitudes can minimize the mean squared error, it indicates deficiency of information to distinguish these attitudes.
We will show that this criterion of observability is the same as the Fisher information criterion if the attitude is assumed to follow matrix Fisher distribution, and the unobserved degree of freedom agrees between the two criteria.

The remainder of this paper is organized as follows.
After reviewing mathematical preliminaries in Section II, we study how the attitude uncertainty evolves in Section III, followed by characterizing the posterior attitude distribution conditioned by a single direction measurement.
Next, we introduce a formal definition of the observability of attitude, and we present the results in \Cref{table:observability}, respectively in Section V and Section VI.
These are followed by numerical simulations, experimental results, and conclusions. 

\section{Mathematical Preliminary}

\subsection{Attitude Kinematics}

Consider the attitude of a rigid body. 
We define an inertial frame $\mathcal{I}=\{\mathbf{e}_1, \mathbf{e}_2, \mathbf{e}_3\}$, where $\mathbf{e}_i$ denotes the vector for the $i$-th axis of $\mathcal{I}$. 
Throughout this paper, we distinguish a vector from its coordinates resolved in a selected basis. 
For example, the coordinates of $\mathbf{e}_i$ in $\mathcal{I}$ is denoted by $e_i\in\Re^{3}$, e.g., $e_1=(1,0,0)$. 
Similarly, we define a body-fixed frame, $\mathcal{B}=\{\mathbf{b}_1, \mathbf{b}_2, \mathbf{b}_3\}$.
The attitude of the rigid body is the orientation of $\mathcal{B}$ relative to $\mathcal{I}$, and it can be defined by a rotation matrix $R\in\Re^{3\times 3}$ defined such that its $i,j$-th element is given by
\begin{align*}
    R_{ij} = \mathbf{e}_i \cdot \mathbf{b}_j.
\end{align*}
The above implies that the $j$-th column of $R$ corresponds to the coordinates of $\mathbf{b}_j$ resolved in $\mathcal{I}$, and the $i$-th row of $R$ corresponds to the coordinates of $\mathbf{e}_i$ resolved in $\mathcal{B}$.
Furthermore, $R$ is the linear transformation of the coordinates of a vector from $\mathcal{B}$ to $\mathcal{I}$. 
As both of $\mathcal{I}$ and $\mathcal{B}$ are right-handed, orthonormal frames, the rotation matrix evolves on the special orthogonal group,
\begin{align*}
    \SO = \{ R\in\Re^{3\times 3}\,|\, R^T R = I_{3\times 3},\; \mathrm{det}[R]=1\}.
\end{align*}

Next, let $\mathbf{w}$ be the angular velocity vector of the rigid body, or equivalently, $ \dot{\mathbf{b}}_i = \mathbf{w}\times \mathbf{b}_i$ for all $i\in\{1,2,3\}$. 
As the coordinates of $\mathbf{b}_i$ in $\mathcal{I}$ is $Re_i$, it implies $\dot R e_i =  \omega\times R e_i$, where $\omega\in\Re^3$ is the coordinates of $\mathbf{w}$ in $\mathcal{I}$.
Thus,
\begin{align}
    \dot R = \hat \omega R,\label{eqn:R_dot_w}
\end{align}
where the hat map $\wedge:\Re^3\rightarrow \so$ is defined such that $\hat x y  = x\times y$ for any $x,y\in\Re^3$, and $\so$ denotes the Lie algebra composed of $3\times 3$ skew-symmetric matrices $\so = \{S\in\Re^{3\times 3}\,|\, S^T = -S\}$. 
Alternatively, let $\Omega = R^T \omega\in\Re^3$ be the coordinates of $\mathbf{w}$ in $\mathcal{B}$.
As $\widehat{Rx} = R x R^T$ for any $R\in\SO$ and $x\in\Re^3$, \eqref{eqn:R_dot_w} can be rewritten as
\begin{align}
    \dot R = R\hat\Omega.\label{eqn:R_dot_W}
\end{align}

Both of \eqref{eqn:R_dot_w} and \eqref{eqn:R_dot_W} ensure that their solutions evolve on $\SO$, or $\dot R$ belongs to the tangent space of $\SO$ at $R$, namely $\T_R\SO$.
Right-multiplying \eqref{eqn:R_dot_w} with $R^T$, we obtain $\hat\omega = \dot R R^T$, which is referred to as \textit{right-trivialization} of the tangent space to $\so$. 
Similarly, $\hat\Omega = R^T \dot R$ is referred to as \textit{left-trivialization}.
Given $\Omega = R^T w$, \eqref{eqn:R_dot_w} is equivalent to \eqref{eqn:R_dot_W}.

\subsection{Matrix Fisher Distribution}\label{sec:MF}

The matrix Fisher distribution is an exponential density formulated for random matrices on Stiefel manifolds~\cite{DowB72,KhaMarJRSSS77}, and stochastic properties of the matrix Fisher distribution specifically on $\SO$ are presented in~\cite{LeeITAC18}.
In this subsection, we summarize selected properties of the matrix Fisher distribution, and we derive additional results required in this paper. 

\begin{definition}{(Matrix Fisher Distribution)}
    A random matrix $R\in\SO$ is distributed according to the matrix Fisher distribution with the matrix parameter $F\in\Re^{3\times 3}$,
    or $R\sim \mathcal{M}(F)$, if its probability density is given by
    \begin{equation}
        p(R) = \frac{1}{c(F)} \exp (\trs{F^T R}),\label{eqn:pMF}
    \end{equation}
    where $c(F) = \int_{R\in\SO} \exp(\trs{F^TR}) dR \in\Re$ denotes the normalizing constant. 
\end{definition}

There are nine free parameters in $F$ to characterize the distribution of three-dimensional attitude. 
The role of $F$ in specifying the shape and dispersion of the distribution can be described after decomposing it into the proper singular value decomposition~\cite{MarJAS88} as follows: 
\begin{align}
	F= U SV^T,\label{eqn:USVp}
\end{align}
where $U,V\in\SO$ and $S\in\Re^{3\times 3}$ is a diagonal matrix $S=\mathrm{diag}[s_1,s_2,s_3]$ with $s_1\geq s_2 \geq |s_3| \geq 0$. 
This is a variation of the common singular value decomposition, defined to ensure $U,V\in\SO$ while allowing the last singular value to be negative. 

First, we formulate the mean attitude. 
Let the first moment of $R\sim\mathcal{M}(F)$ be denoted by $\fm[F] = \E[R] \in\Re^{3\times 3}$.
According to~\cite{LeeITAC18}, we have
\begin{align}
    \fm[F] =   U \left( \frac{1}{c(S)} \mathrm{diag}\!\left [ \deriv{c(S)}{s_1}, \deriv{c(S)}{s_2}, \deriv{c(S)}{s_3} \right] \right) V^T.\label{eqn:M}
\end{align}
As shown above, the arithmetic mean $\mathrm{E}[R]$ does not necessarily belong to $\SO$. 
As such, the \textit{mean attitude} is formulated as the attitude that maximizes the density function or that minimizes the Frobenius mean squared error \cite{moakher2002means}, given by
\begin{align*}
    \mathrm{M}[F] = UV^T,
\end{align*}
which is further discussed in \Cref{lem:MMSE} later.

Next, we define the principal axes. 
Consider rotating the mean attitude by an angle $\theta_i\in[0,2\pi)$ as follows:
\begin{align*}
   	R(\theta_i) & = U \exp(\theta_i\hat e_i) V^T \\
                   &= UV^T \exp(\theta_i \widehat{Ve_i}) = \exp(\theta_i \widehat{Ue_i}) UV^T.
\end{align*}
All of the above expressions are equivalent: they correspond to the rotation of the mean attitude $UV^T$ by the angle $\theta_i$ about the axis, which is $Ue_i$ when resolved in the inertial frame, or equivalently $Ve_i= (UV^T)^T Ue_i$  when resolved in the body-fixed frame of the mean attitude. 
Then we have
\begin{align} \label{eqn:pa_density}
	p(R(\theta_i)) = \frac{ e^{s_i} }{c(S)} \exp((s_j+s_k)\cos\theta_i),
\end{align}
where $(i,j,k)\in\{(1,2,3),(2,3,1),(3,1,2)\}$. 
As such, the rotation about the $i$-th axis is more concentrated as $s_j+s_k$ increases, and
this axis is referred to as the $i$-th principal axis of the distribution. 

In summary, the interpretation of the matrix parameter $F=USV^T$ in specifying the shape and dispersion of the matrix Fisher distribution is as follows: 
\begin{itemize}
    \item $UV^T\in\SO$: mean attitude
    \item $U\in\SO$: linear transformation of the coordinates of a vector from the principal axes frame to the inertial frame
    \item $V\in\SO$: linear transformation of the coordinates of a vector from the principal axes frame to the body-fixed frame of the mean attitude
    \item $s_1,s_2,s_3\in\Re^3$: degree of concentration; the rotation about the $i$-th principal axis is more concentrated as $s_j+s_k$ is increased.
\end{itemize} 

Next, we present the maximum likelihood estimation (MLE) for the matrix Fisher distribution.
For a given set of samples $\{R_i\}_{i=1}^N$, let $\bar R$ be the arithmetic mean or the first moment. 
The corresponding MLE of the parameter is given by $F = \fm^{-1}[\bar R]$, where the inverse of $\fm$ is defined as follows. 
Let $UDV^T$ be the proper singular value decomposition of $\bar{R}$, where $D=\mathrm{diag}[d_1,d_2,d_3]\in\Re^{3\times 3}$.
The corresponding diagonal matrix $S=\mathrm{diag}[s_1,s_2,s_3]$ is constructed by solving
\[
    \frac{1}{c(S)}\frac{\partial c(S)}{\partial s_i} = d_i,
\]
for $(s_1,s_2,s_3)$. 
Then $F = \fm^{-1}[\bar{R}] = USV^T$.
In other words, we can construct a matrix Fisher distribution from the first moment of a random rotation matrix using the MLE.

Finally, we present a relationship between the concentration parameter $S$, and the matrix $D$ associated with the first moment.

\begin{lem} \label{lemma:SD}
    Suppose a random rotation matrix $Q\in\SO$ is distributed according to $Q\sim\mathcal{M}(S)$ for
    $S = \mathrm{diag}[s_1,s_2,s_3]$.
    Based on \eqref{eqn:M}, $\E[Q]$ is diagonal as $S$ is. 
    Let $\E[Q] = D = \mathrm{diag}[d_1,d_2,d_3]$.
    Then, the following properties hold:
    \begin{enumerate}
        \item $s_i+s_j=0$ if and only if $d_i+d_j=0$; 
        \item $s_i=s_j=0$ if and only if $d_i=d_j=0$;
        \item $d_i+d_j$ is monotonically increasing with $s_i+s_j$, 
    \end{enumerate}
    for any $(i,j,k)\in\{(1,2,3),(2,3,1),(3,1,2)\}$.
\end{lem}
\begin{proof}
	From Theorem 2.1 in \cite{LeeITAC18}, we have 
	\begin{align}
		&\frac{\partial c(S)}{\partial s_i} + \frac{\partial c(S)}{\partial s_j} = \int_{-1}^{1} \frac{1}{2}(1+u)I_0\left[\frac{1}{2}(s_i-s_j)(1-u)\right] \nonumber \\
		&\qquad \times I_1\left[\frac{1}{2}(s_i+s_j)(1+u)\right] \exp(s_ku)\, du \geq 0, \label{eqn:dcdsi+j}
	\end{align}
	where $I_\nu(x) = \int_{0}^{2\pi} \cos(\nu\theta)\exp(x\cos\theta) d\theta$ is the modified Bessel function of the first kind of order $\nu$.
    In particular, $I_0(x) \geq 1$ for any $x$, and $I_1(x)=0$ if and only if $x=0$.
    Also $I_1(x) >0$ if $x>0$.
    
    First, if $s_i+s_j=0$, then $I_1\left[\frac{1}{2}(s_i+s_j)(1+u)\right] = 0$, and therefore $d_i+d_j = \frac{1}{c(S)}\left(\frac{\partial c(S)}{\partial s_i} + \frac{\partial c(S)}{\partial s_j}\right) = 0$.
    On the other hand, if $s_i+s_j>0$, the integrand of \eqref{eqn:dcdsi+j} is strictly positive for $u\in(-1,1]$.
	Therefore, $d_i+d_j>0$.
    As both of $d_i+d_j$ and $s_i+s_j$ are non-negative, it follows that $d_i+d_j =0$ implies $s_i+s_j=0$. 
    Therefore, $s_i+s_j = 0$ if and only if $d_i+d_j=0$.
	
	Next, if $s_i=s_j=0$, then by (15a) and (15b) in \cite{LeeITAC18}, $d_i=d_j=0$.
	On the other hand, suppose at least one of $s_i$ or $s_j$ is nonzero. There are two sub-cases: (i) if $s_i+s_j\neq 0$, then $d_i+d_j\neq 0$ so $d_i$, $d_j$ cannot both be zeros;
	(ii) if $s_i = -s_j \neq 0$, then
	\begin{align*}
		\frac{\partial c(S)}{\partial s_i} = \int_{-1}^1 \frac{1}{4}(1-u) I_1[s_i(1-u)]\exp(s_ku)du
	\end{align*}
    where the integrand has the same sign as $s_i$ over $u\in[-1,1)$.
	So $d_i\neq 0$.
	
    Finally to prove the monotonicity, let $\alpha=s_i+s_j$ and $\beta = s_i-s_j$, which yield $s_i = \frac{1}{2}(\alpha+\beta)$ and $s_j = \frac{1}{2}(\alpha-\beta)$.
	From the chain rule,
	\begin{align}
        \frac{\partial (d_i+d_j)}{\partial \alpha} &= \frac{1}{2} \braces{\deriv{(d_i+d_j)}{s_i} + \deriv{(d_i+d_j)}{s_j}}.\label{eqn:dij_alpha}
    \end{align}
    Since
    \begin{align*}
        d_i + d_j  = \frac{1}{c(S)}\int_{Q\in\SO}(Q_{ii}+Q_{jj}) \exp(\trs{SQ}) dQ,
    \end{align*}
    the above reduces to
    \begin{align*}
        \frac{\partial (d_i+d_j)}{\partial \alpha} & = \frac{1}{2} (\E[(Q_{ii}+Q_{jj})^2] - E[Q_{ii}+Q_{jj}]^2 )
	\end{align*}
	which is positive by the Cauchy-Schwarz inequality.
    Therefore, the monotonicity follows.
\end{proof}

\section{Attitude Uncertainty Propagation}\label{sec:UP}

In this section, we consider stochastic versions of the attitude kinematics equations \eqref{eqn:R_dot_w} and \eqref{eqn:R_dot_W}, where the rotation matrix $R$ is considered as a $\SO$-valued random process, and we study how the uncertainty distribution of $R$ evolves over time. 

\subsection{Stochastic Attitude Kinematics}

The stochastic attitude kinematics equation corresponding to \eqref{eqn:R_dot_w} is given by
\begin{align}
    dR = (\omega(t) dt + H(t) dW)^\wedge R,\label{eqn:dR_w}
\end{align}
where the angular velocity $\omega(t)$ resolved in $\mathcal{I}$ is assumed to be given deterministically as a function of time with an additive noise $H(t)d W$,
which is a Wiener process $W\in\Re^3$ scaled by a matrix $H(t)\in\Re^{3\times 3}$. 
The above stochastic differential equation is defined according to the Stratonovich sense so that the random matrix $R$ evolves on $\SO$ \cite{barrau2018stochastic}. 

Next, the stochastic differential equation corresponding to \eqref{eqn:R_dot_W} is
\begin{align}
    dR = R (\Omega(t) dt + H(t) dW)^\wedge ,\label{eqn:dR_W}
\end{align}
with the angular velocity $\Omega(t)$ resolved in $\mathcal{B}$. 
We call \eqref{eqn:dR_w} and \eqref{eqn:dR_W} right-trivialized and left-trivialized, respectively. 
In the deterministic case, the right-trivialized \eqref{eqn:R_dot_w} and the left-trivialized \eqref{eqn:R_dot_W} are equivalent.
One of the fundamental questions of this paper is whether such equivalence holds in the stochastic case. 
Unfortunately, there is no explicit, analytical solution to either \eqref{eqn:dR_w} or \eqref{eqn:dR_W}. 
Instead, we study how the first moment of $R$ evolves, and interpret its uncertainty utilizing the matrix Fisher distribution.

\subsection{Propagation of First Moments}

First, consider the right-trivialized kinematics \eqref{eqn:dR_w}.
A given first moment $\E[R(t)]$ at $t$ is propagated to $\tau>t$ as follows, with the angular velocity trajectory $\omega(\cdot)$ between $t$ and $\tau$.

\begin{theorem}\label{thm:ER_R}
    Suppose $R$ follows the right-trivialized stochastic differential equation \eqref{eqn:dR_w}.
    For a given $\E[R(t)]$ at $t$, the first moment $\E[R(\tau)]$ at $\tau$ is written as
    \begin{align}
        \E[R(\tau)] & = \Phi_R(\tau, t) \E[R(t)]  + \mathcal{O}((\tau-t)^{2}),\label{eqn:ER_tau_R}
    \end{align}
    where $\Phi_R(\tau,t)\in\Re^{3\times 3}$ is
    \begin{align}
        \Phi_R(\tau,t) & = \exp(\hat\phi_R(\tau,t)) \nonumber\\
        & \quad \times \{I_{3\times 3} + \frac{1}{2}(G_R(\tau,t) - \trs{G_R(\tau,t)}I_{3\times 3})\} .\label{eqn:Phi_R}
    \end{align} 
    And $\hat\phi_R(\tau,t)$ and $G_R(\tau, t)\in\Re^{3\times 3}$ are
    \begin{align}
        \hat\phi_R(\tau, t) & = \int_{t}^\tau \dexp^{-1}_{\hat\phi_R(\sigma, t)} \hat\omega(\sigma) d\sigma, \label{eqn:phi_R}\\
        G_R(\tau, t) & =\int_{t}^\tau \exp(-\hat\phi_R(\sigma,t)) H(\sigma) H^T(\sigma) \exp(\hat\phi_R(\sigma,t))  |d\sigma|.\label{eqn:G}
    \end{align}
\end{theorem}
\begin{proof}
    See Appendix \ref{app:ER}.
\end{proof}

In \eqref{eqn:ER_tau_R}, propagation of the first moment from $t$ to $\tau$ corresponds to left-multiplying $\Phi_R(\tau,t)$ with accuracy up to the first order of $|\tau -t|$, resembling the state transition matrix in linear system dynamics.
Looking at \eqref{eqn:Phi_R}, $\Phi_R$ is composed of two parts: the first exponential term corresponds to the advection, or the rotation of the distribution, due to the angular velocity $\omega$; and the second part in the braces represents the diffusion due to the Wiener process. 

Alternatively, the first moment $\E[R(t)]$ can be propagated over the left-trivialized kinematics \eqref{eqn:dR_W} as follows. 
\begin{theorem}\label{thm:ER_L}
    Suppose $R$ follows the left-trivialized stochastic differential equation \eqref{eqn:dR_W}.
    For a given $\E[R(t)]$ at $t$, the first moment $\E[R(\tau)]$ at $\tau$ is written as
    \begin{align}
        \E[R(\tau)] & = \E[R(t)] \Phi_L(\tau, t) + \mathcal{O}((\tau-t)^{2}),\label{eqn:ER_tau_L}
    \end{align}
    where $\Phi_L(\tau,t)\in\Re^{3\times 3}$ is
    \begin{align}
        \Phi_L(\tau,t) & = \{I_{3\times 3} + \frac{1}{2}(G_L(\tau,t) - \trs{G_L(\tau,t)}I_{3\times 3})\}\nonumber\\
                       & \quad \times \exp(\hat\phi_L(\tau,t)).\label{eqn:Phi_L}
    \end{align} 
    And $\hat\phi_L(\tau,t)$ and $G_L(\tau, t)\in\Re^{3\times 3}$ are
    \begin{align}
        \hat\phi_L (\tau, t) & = \int_{t}^\tau \dexp^{-1}_{-\hat\phi_L(\sigma, t)} \hat\Omega(\sigma) d\sigma, \label{eqn:phi_L}\\
        G_L(\tau, t) & =\int_{t}^\tau \exp(\hat\phi_L(\sigma,t)) H(\sigma) H^T(\sigma) \exp(-\hat\phi_L(\sigma,t))  |d\sigma|.\label{eqn:G_L}
    \end{align}
\end{theorem}
\begin{proof}
    The proof is similar to \Cref{thm:ER_R}, and is omitted for brevity.
\end{proof}
Equation \eqref{eqn:ER_tau_L} is analogous to \eqref{eqn:ER_tau_R}, except that $\Phi_L(\tau,t)$ is acting on the right. 
The interpretation of $\Phi_L(\tau,t)$ in \eqref{eqn:Phi_L} is also similar with \eqref{eqn:Phi_R}: it represents an advection and a diffusion.

\subsection{Differences Between Right- and Left-Trivialization} \label{sec:RL}

\Cref{thm:ER_R} and \Cref{thm:ER_L} describe how the first moment evolves through the right-trivialized \eqref{eqn:dR_w}, and the left-trivialized \eqref{eqn:dR_W}, respectively. 
To understand the difference between them more intuitively, we consider a simpler case when the angular velocity is discretized by a zero-order hold.
Let the time be discretized by a sequence $\{t_k\}_{k=0}^N $ with a fixed time step $h=t_{k+1}-t_k$, and let $\omega(t) = \omega_k$ for any $t\in[t_k,t_{k+1})$. 
Under these assumptions, we can show
\begin{align*}
    \phi_R (t_{k+1}, t_k ) & = h\omega_k,\\
    G_R(t_{k+1}, t_k ) & = h H_k H_k^T,
\end{align*}
where the higher-order terms of $h$ are omitted (see Appendix \ref{app:ER_sim}). 
Further assume $H_k = \gamma I_{3\times 3}$ for $\gamma \in\Re$.
Then, \eqref{eqn:ER_tau_R} is rewritten as
\begin{align}
    \E[R_{k+1}] = (1-h\gamma^2) \exp(h\hat\omega_k) \E[R_{k}],\label{eqn:ER_kp_R}
\end{align}
where the first moment is reduced by the factor $1-h\gamma^2$ and it is rotated about the axis $h\omega_k$.

This is interpreted using the matrix Fisher distribution as follows. 
Suppose $R_k \sim \mathcal{M}(F_k)$ with $F_k = U_k S_k V_k^T\in\Re^{3\times 3}$.
From \eqref{eqn:M}, its first moment is given by $\E[R_k] = U_k D_k V_k^T$, where $D_k = \mathrm{diag}[d_{1},d_{2}, d_{3}]$ with $d_{i} = \frac{1}{c(S_k)}\deriv{c(S_k)}{s_{i}}$ for $i=1,2,3$.
Therefore, \eqref{eqn:ER_kp_R} is rewritten as 
\begin{align*}
    \E[R_{k+1} ] =\exp(h\hat\omega_k) U_k \times (1-h\gamma^2) D_k \times V_k^T,
\end{align*}
which  is already given in a form of the proper singular value decomposition.
Thus, by taking $\fm^{-1}$ of the above, we have $R_{k+1}\sim\mathcal{M}(F_{R_{k+1}})$ with the matrix parameter $F_{R_{k+1}}\in\Re^{3\times 3}$ given by
\begin{align} \label{eqn:F_R_kp}
    F_{R_{k+1}} = U_{R_{k+1}} S_{R_{k+1}} V_{R_{k+1}}^T,
\end{align}
where $U_{R_{k+1}}, V_{R_{k+1}}\in\SO$ are 
\begin{gather}
    U_{R_{k+1}} = \exp(h\hat\omega_k) U_k,\; V_{R_{k+1}} = V_k,\label{eqn:UV_R_kp}
\end{gather}
and the diagonal $S_{R_{k+1}}\in\Re^{3\times 3}$ is defined such that its diagonal elements $s_{R_i}$ satisfy
\begin{align}
    \frac{1}{c(S_{R_{k+1}})} \deriv{ c(S_{R_{k+1}})}{ s_{R_i} } = (1-h\gamma^2) \frac{1}{c(S_k)} \deriv{ c(S_k) }{ s_i},\label{eqn:S_R_kp}
\end{align}
for $i\in\{1,2,3\}$.
Note that here we are matching the true distribution of $R_{k+1}$ to a matrix Fisher distribution using the MLE.

Now, let us compare  $R_k\sim\mathcal{M}(F_k)$ and $R_{k+1}\sim\mathcal{M}(F_{R_{k+1}})$ to understand how the attitude uncertainties are propagated over the right-trivialized \eqref{eqn:dR_w}.
In particular, we focus on the three aspects of the mean attitude, the degree of dispersion, and the principal axes. 
First, the mean attitude is rotated from $\mathrm{M}[F_k] = U_k V_k^T$ into
\begin{align}
    \mathrm{M}[F_{R_{k+1}} ] = U_{R_{k+1}} V_{R_{k+1}}^T = \exp(h\hat\omega_k ) U_k V_k^T ,\label{eqn:M_kp_R}
\end{align}
by the rotation vector $h\omega_k$ resolved in the inertial frame, which is expected.
Next, the uncertainty becomes more dispersed, as $S_{R_{k+1}}$ is reduced from $S_k$ according to \eqref{eqn:S_R_kp} and \Cref{lemma:SD}.
In case there is no noise or diffusion, i.e., $\gamma = 0$, the degree of dispersion remains unchanged as $S_{R_{k+1}} = S_k$.
Finally, the effect of advection on the principal axes are described by \eqref{eqn:UV_R_kp}. 
Since $U_{R_{k+1}} = \exp(h\hat\omega_k) U_k$, the principal axes are rotated by the rotation vector $h\omega_k$ when perceived in the inertial frame.
However, as $V_{R_{k+1}} = V_k$, the principal axes remain unchanged when observed from the body-fixed frame. 
In other words, the shape of uncertainties remains unchanged relative to the body-fixed frame. 
But, as the body-fixed frame is rotated by $h\omega_k$, the distribution is also rotated accordingly in the inertial frame. 

Then, what happens if the left-trivialized \eqref{eqn:dR_W} is used instead?
When $R_k\sim\mathcal{M}(F_k)$ is propagated through \eqref{eqn:dR_W}, we can similarly show that $R_{k+1} \sim \mathcal{M}(F_{L_{k+1}})$ with
\begin{align}
    F_{L_{k+1}} = U_{L_{k+1}} S_{L_{k+1}} V_{L_{k+1}}^T, \label{eqn:F_L_kp}
\end{align}
where $U_{L_{k+1}}, V_{L_{k+1}}\in\SO$ are given by
\begin{gather}
    U_{L_{k+1}} = U_k,\; V_{L_{k+1}} = \exp(-h\hat\Omega_k) V_k,\label{eqn:UV_L_kp}
\end{gather}
and the diagonal $S_{L_{k+1}}\in\Re^{3\times 3}$ is defined the same as \eqref{eqn:S_R_kp}.
First, the mean attitude is
\begin{align}
    \mathrm{M}[F_{L_{k+1}} ] = U_{L_{k+1}} V_{L_{k+1}}^T =  U_k V_k^T \exp(h\hat\Omega_k ),\label{eqn:M_kp_L}
\end{align}
which is obtained by rotating the prior mean $\mathrm{M}[F_k]= U_kV_k^T$ by the vector $h\Omega_k$ resolved in the body-fixed frame. 
In fact, this is equivalent to \eqref{eqn:M_kp_R}, when $\Omega_k = (U_kV_k^T)^T \omega_k$, i.e., they are transformed to each other by the mean attitude.
Next, the effect of diffusion is the same as the right-trivialized case: the uncertainties become more dispersed when $\gamma>0$, and the dispersion is unchanged when $\gamma=0$. 
Finally, while the mean attitude and the degree of dispersion are propagated in the same manner, the principal axes are transformed in the opposite manner.
As $U_{L_{k+1}} = U_k$, the principal axes remain unchanged when observed from the inertial frame. 
But, since $V_{L_{k+1}} = \exp(-h\hat\Omega_k) V_k$, the principal axes are rotated by the rotation vector $-h\Omega_k$ when perceived in the body-fixed frame. 
In other words, the shape of uncertainties represented by the most uncertain rotation or least uncertain rotation remains unchanged relative to the inertial frame.
But, as the body-fixed frame is rotated by $h\Omega_k$, the principal axes are rotated in the opposite way in the body-fixed frame, to keep them unchanged in the inertial frame. 

In short, in the deterministic attitude kinematics, the order of trivialization does not matter, as \eqref{eqn:R_dot_w} is completely equivalent to \eqref{eqn:R_dot_W}. 
However, the equivalence between right and left does not hold in the stochastic attitude kinematics: while the mean attitude and the degree of dispersion are changed in the same way for both of \eqref{eqn:dR_w} and \eqref{eqn:dR_W}, the right-trivialized \eqref{eqn:dR_w} rotates the principal axes in the inertial frame, and the left-trivialized \eqref{eqn:dR_W} rotates them in the body-fixed frame. 
For \eqref{eqn:dR_w}, the principal axes remain unchanged when observed from the body-fixed frame; 
and for \eqref{eqn:dR_W}, they remain unchanged relative to the inertial frame. 
This distinction is entirely attributable to the advection of the prior distribution for $R_k$, and has nothing to do with the diffusion when the noise is isotropic.

\subsection{Numerical Example}

To illustrate these more clearly, we present a numerical example. 
Initially, $R_0 \sim \mathcal{M}(F_0)$ with $F_0 = \mathrm{diag}[150, 10, 0]$, which implies
\begin{align*}
    U_0 = V_0 = I_{3\times 3}, \quad S_0 = \mathrm{diag}[150, 10, 0].
\end{align*}
The initial mean attitude is $\mathrm{M}[F_0]=I_{3\times 3}$, and the first principal axis is $U_0 e_1 = e_1$ when resolved in the inertial frame.
Thus, it is the first base vector $\mathbf{e}_1$ of the inertial frame.
As discussed in \Cref{sec:MF}, since $s_2+ s_3 \leq s_3+s_2 \leq s_1+s_2$, the distribution is most uncertain for the rotation about the first principal axis. 
The corresponding initial distribution is illustrated in \Cref{fig:prop_R_1}, where the three body-fixed axes of the mean attitude are shown by the red (first), green (second), and blue (third) arrows, respectively. 
We also compute the marginal distribution of each body-fixed axis, and visualize it on the unit-sphere via color shading~\cite{LeeITAC18}.
As the first principal axis is $\mathbf{e}_1$, in \Cref{fig:prop_R_1}, the distributions for the second body-fixed axis (green) and the third body-fixed axis (blue) are elongated along the great circle normal to $\mathbf{e}_1$. 
The distributions along the other principal axes $\mathbf{e}_2,\mathbf{e}_3$ are relatively narrower.

We propagate the initial distribution $\mathcal{M}(F_0)$ with a fixed angular velocity $\omega = \frac{\pi}{2}e_3\, \si{\radian\per\second} $ until $t=\SI{1}{\second}$, without any noise.
The four sub-figures in the left column of \Cref{fig:prop} show the distributions propagated by the right-trivialized \eqref{eqn:dR_w}.
First, as time increases, the mean attitude rotates about $\mathbf{e}_3$ according to \eqref{eqn:M_kp_R}, by up to \SI{90}{\degree} at $t=\SI{1}{\second}$.
Next, as discussed above, the principal axes also rotate in the same manner.
For example, the first principal axis rotates from $\mathbf{e}_1$ at $t=\SI{0}{\second}$ to $\mathbf{e}_2$ at $t=\SI{1}{\second}$.
As the principal axes are fixed in the body-fixed frame and the degree of dispersion remains unchanged, it is as if the sphere at \Cref{fig:prop_R_1} literally rotates to \Cref{fig:prop_R_4}, without altering the distribution shown by color marks. 

Next, we propagate the initial distribution with the left-trivialized \eqref{eqn:dR_W}. 
The angular velocity is obtained by transforming the above $\omega$ to the body-fixed frame using the initial mean attitude. 
As the mean attitude is $I_{3\times 3}$, this trivially yields $\Omega= \omega$. 
The four sub-figures in the right column of \Cref{fig:prop} show the corresponding propagated distributions.
The mean attitude rotates in the same way as the right-trivialized case, and the degree of dispersion is unchanged. 
But, the key difference is that the principal axes do not rotate in the inertial frame. 
For instance, the first principal axis is $\mathbf{e}_1$ always.  
Therefore, the distribution of each axis of the mean attitude is always elongated about the corresponding great circle normal to $\mathbf{e}_1$. 
As a result, at the terminal time given in \Cref{fig:prop_L_4}, the second body-fixed axis (green) becomes most concentrated as it is parallel to $\mathbf{e}_1$,
and the distributions of the first body-fixed axis (red) and the third (blue) are elongated along the great circle in the $\mathbf{e}_2$--$\mathbf{e}_3$ plane. 
In summary, the difference between \eqref{eqn:dR_w} and \eqref{eqn:dR_W} is well-illustrated by how the uncertainties are distributed about the mean attitude in \Cref{fig:prop_R_4} and \Cref{fig:prop_L_4}. 

\begin{figure}
	\centerline{
		\subfigure[right-triv. \eqref{eqn:dR_w} at $t=\SI{0}{\second}$]{
			\includegraphics[width=0.45\columnwidth, trim=150 80 130 60, clip]{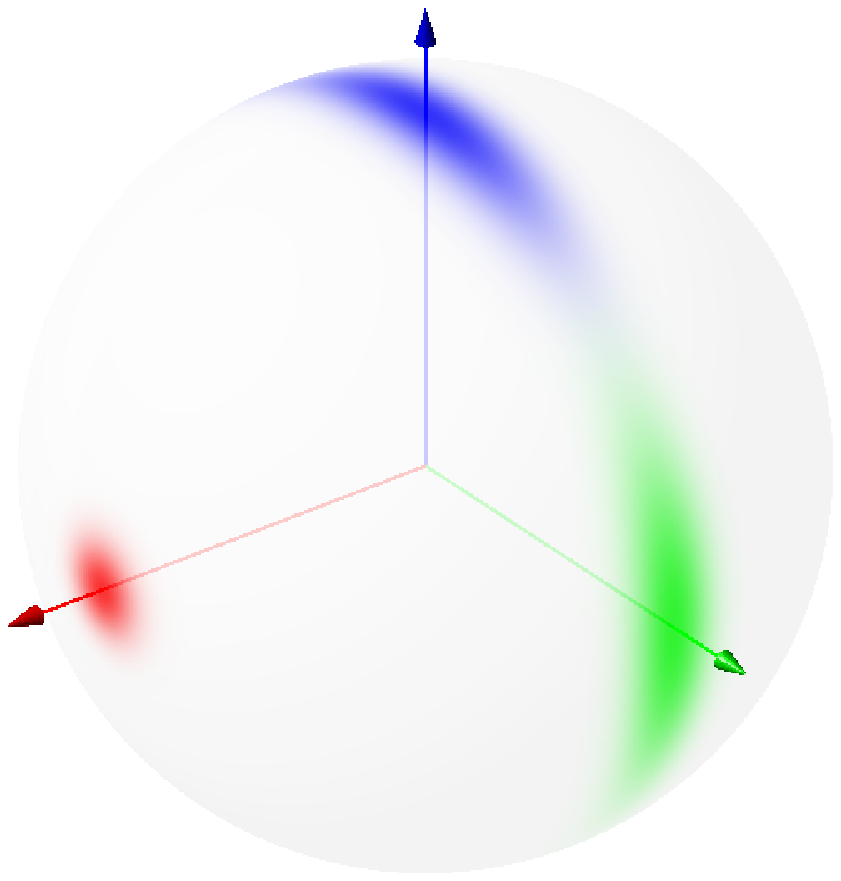}
			\begin{tikzpicture}[overlay]
			\draw[arrows={-Triangle[angle=30:4pt]}] (-2.9,0.25) -- ++(90:0.5);
			\draw[arrows={-Triangle[angle=30:4pt]}] (-2.9,0.25) -- ++(-30:0.5);
			\draw[arrows={-Triangle[angle=30:4pt]}] (-2.9,0.25) -- ++(210:0.5);
			\node at (-3.3,0.2) {\scriptsize $\mathbf{e}_1$};
			\node at (-2.5,0.2) {\scriptsize $\mathbf{e}_2$};
			\node at (-2.9,0.85) {\scriptsize $\mathbf{e}_3$};
			\end{tikzpicture}
			\label{fig:prop_R_1}
		}
		\subfigure[left-triv. \eqref{eqn:dR_W} at $t=\SI{0}{\second}$]{
			\includegraphics[width=0.45\columnwidth, trim=150 80 130 60, clip]{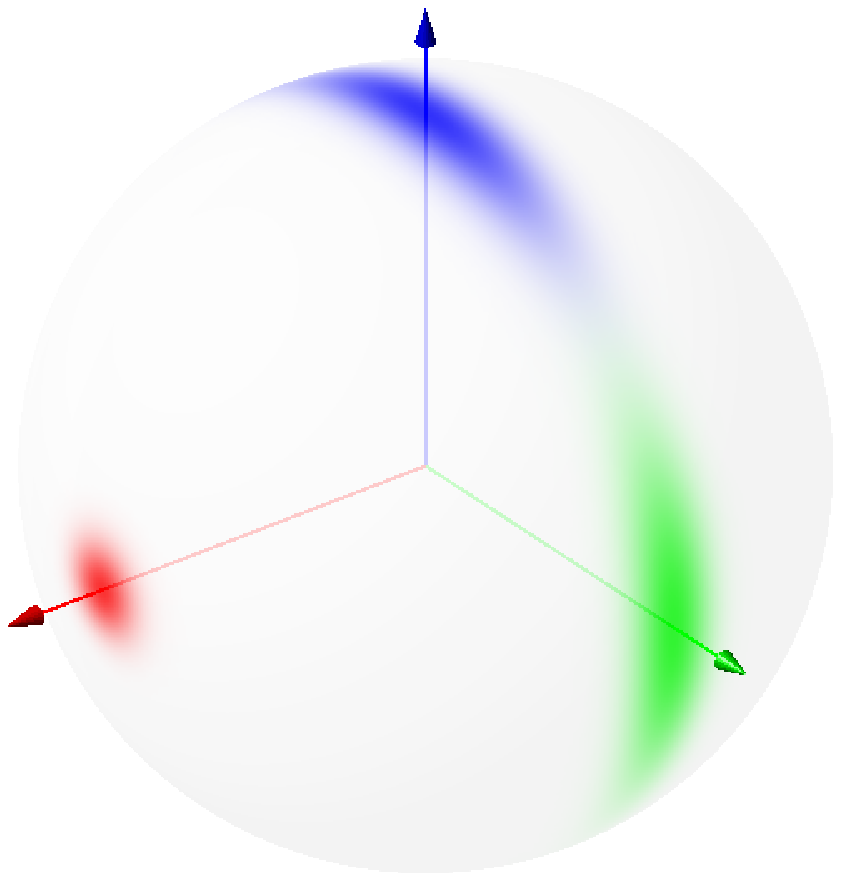}
		}
	}
	\centerline{
		\subfigure[right-triv. \eqref{eqn:dR_w} at $t=\SI{0.33}{\second}$]{
			\includegraphics[width=0.45\columnwidth, trim=150 80 130 60, clip]{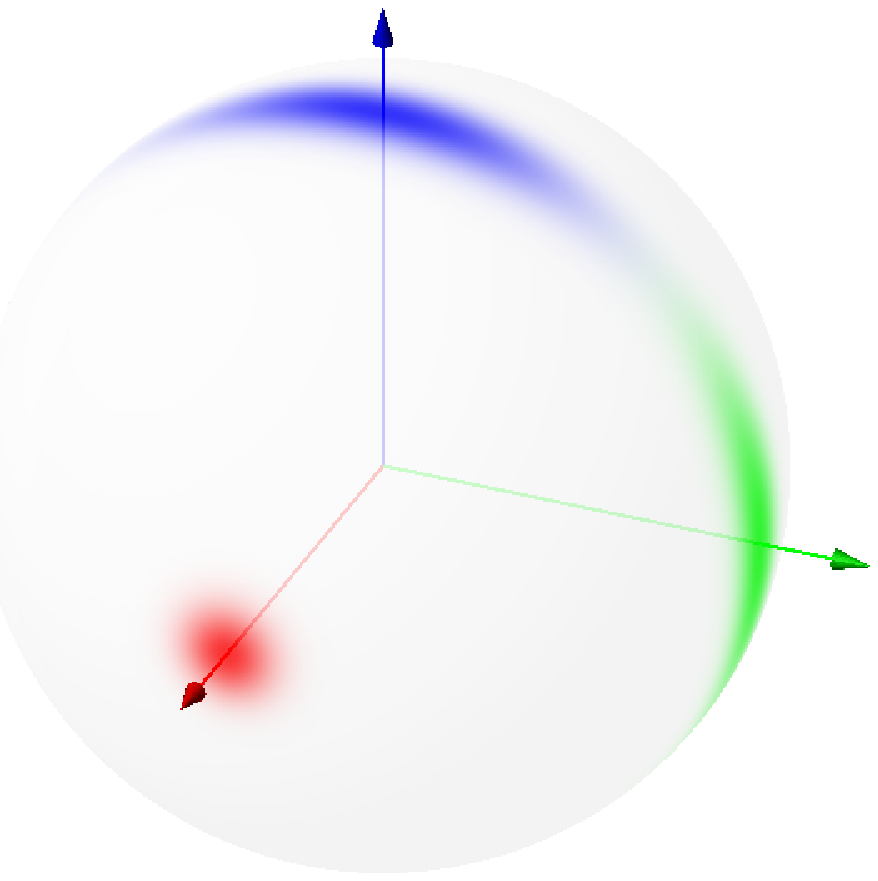}
		}
		\subfigure[left-triv. \eqref{eqn:dR_W} at $t=\SI{0.33}{\second}$]{
			\includegraphics[width=0.45\columnwidth, trim=150 80 130 60, clip]{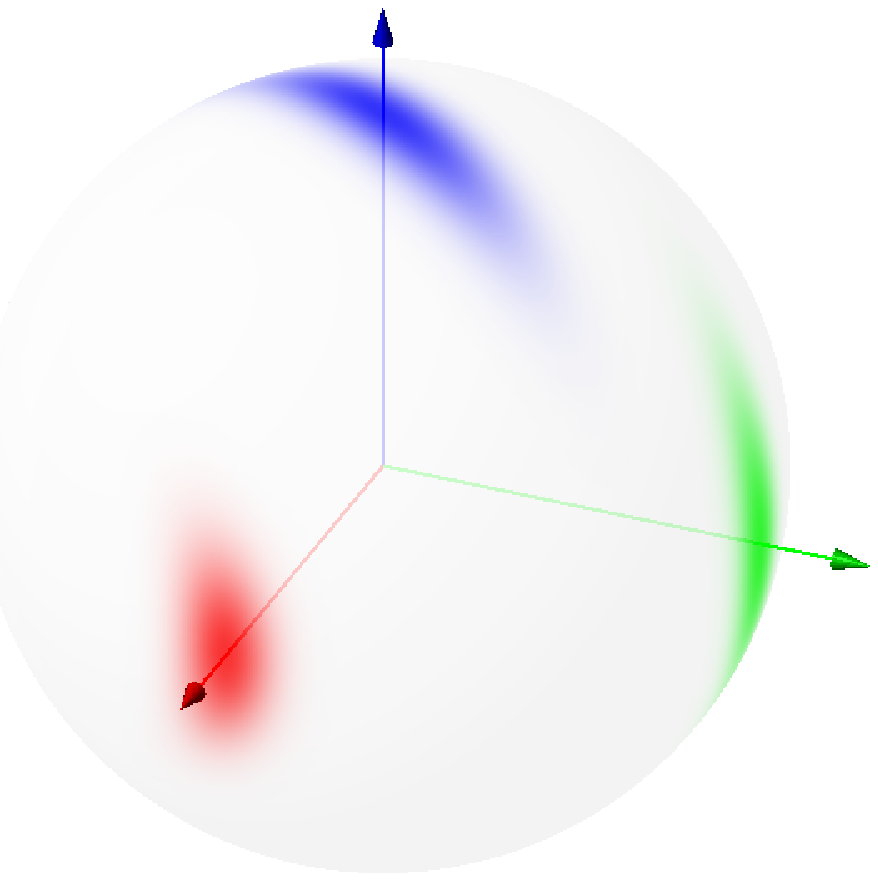}
		}
	}
	\centerline{
		\subfigure[right-triv. \eqref{eqn:dR_w} at $t=\SI{0.67}{\second}$]{
			\includegraphics[width=0.45\columnwidth, trim=150 80 130 60, clip]{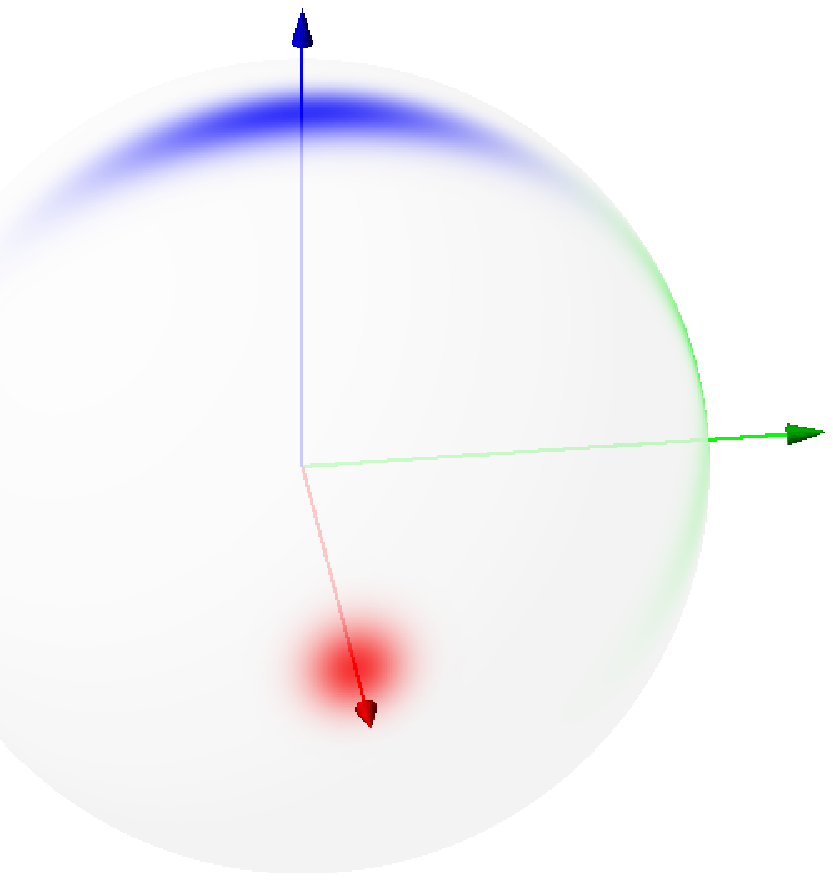}
		}
		\subfigure[left-triv. \eqref{eqn:dR_W} at $t=\SI{0.67}{\second}$]{
			\includegraphics[width=0.45\columnwidth, trim=150 80 130 60, clip]{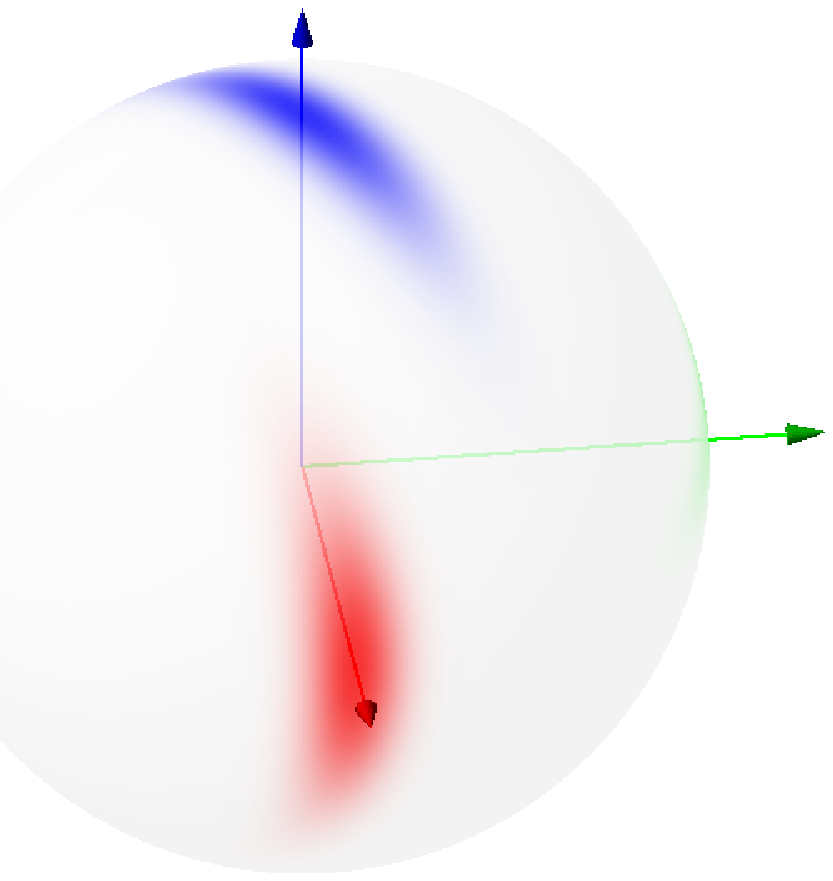}
		}
	}
	\centerline{
		\subfigure[right-triv. \eqref{eqn:dR_w} at $t=\SI{1}{\second}$]{
			\includegraphics[width=0.45\columnwidth, trim=150 80 130 60, clip]{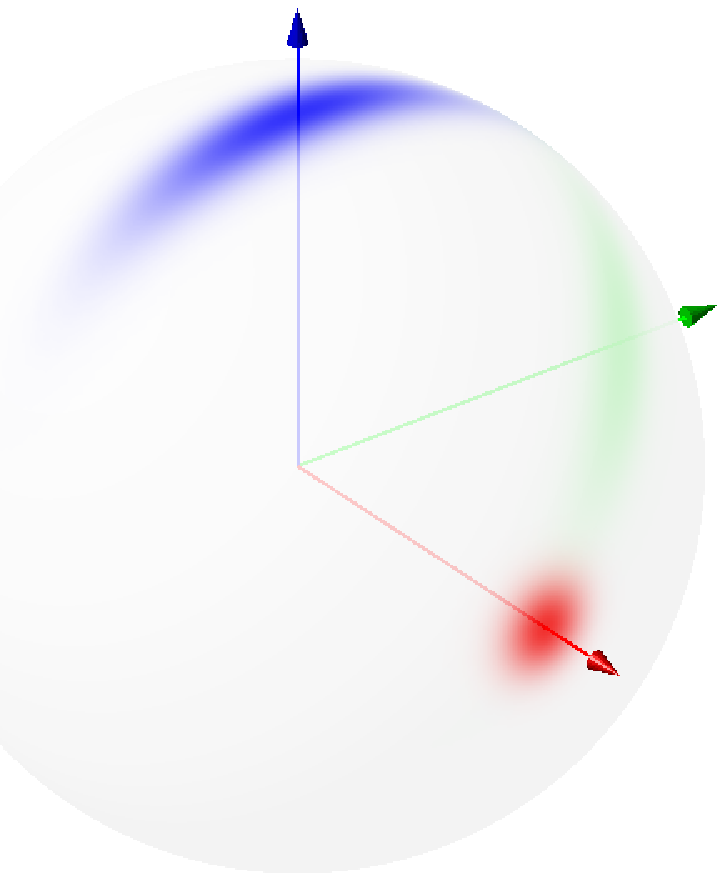}
			\label{fig:prop_R_4}
		}
		\subfigure[left-triv. \eqref{eqn:dR_W} at $t=\SI{1}{\second}$]{
			\includegraphics[width=0.45\columnwidth, trim=150 80 130 60, clip]{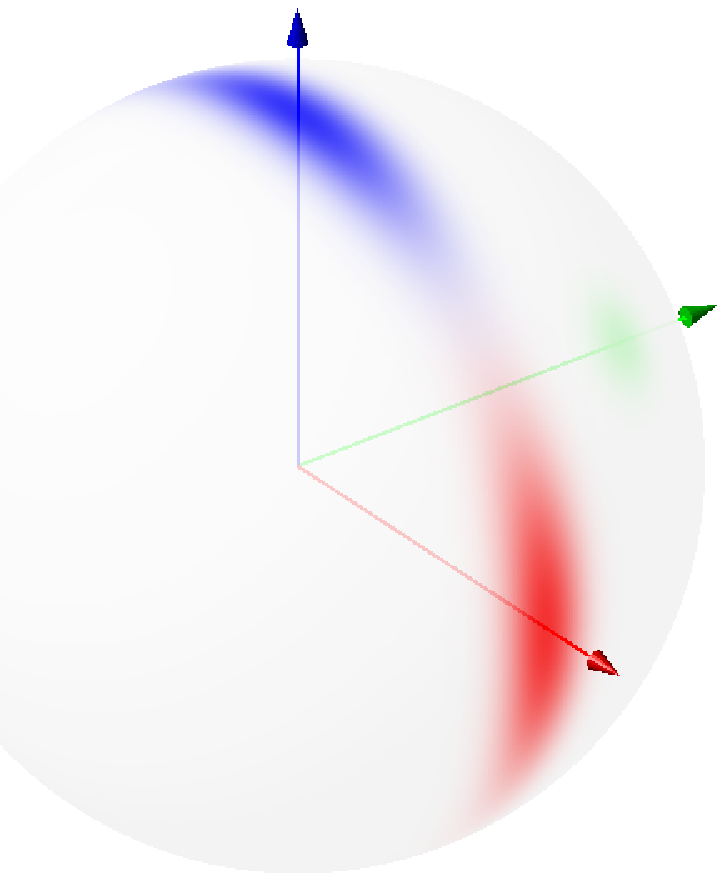}
			\label{fig:prop_L_4}
		}
	}
	\caption{Illustration of propagated uncertainties: (a,c,e,g) by the right-trivialized \eqref{eqn:dR_w}; (b,d,f,h) by the left-trivialized \eqref{eqn:dR_W}.}\label{fig:prop}
\end{figure}

\subsection{Intuition}

So far, we have presented a theoretical analysis and a numerical example illustrating how uncertainties are propagated through \eqref{eqn:dR_w} and \eqref{eqn:dR_W} differently. 
This section is concluded by providing an intuition behind the discrepancy. 
In the deterministic case, the key relation to associate \eqref{eqn:R_dot_w} with \eqref{eqn:R_dot_W} is $\Omega = R^T \omega$. 
Whereas in the stochastic case, such transformation between $\Omega$ and $\omega$ does not exist. 
Suppose that the initial distribution is represented by a large number of sample attitudes. 
For the right-trivialized \eqref{eqn:dR_w}, at any instance, all of sample attitudes are subject to the same angular velocity, whose coordinates in the inertial frame are given by $\omega(t)$. 
This explains why the distribution \textit{purely rotates} in the left column of \Cref{fig:prop}, when observed from the inertial frame. 
On the other hand, the angular velocity $\Omega(t)$ is prescribed in the body-fixed frame at \eqref{eqn:dR_W}. 
While all of the sample attitudes share the same $\Omega(t)$, they are rotated by non-identical angular velocity \textit{vectors} as they have different body-fixed axes in the inertial frame. 
This is the fundamental reason why \eqref{eqn:dR_W} is distinct from \eqref{eqn:dR_w}: there is no $\Omega(t)$ resulting in the same rotation for all sample attitudes when observed from the inertial frame. 

\section{Single Direction Measurements}\label{sec:SDM}

In this section, we present two types of direction measurements, and show how these measurements differ in characterizing attitude uncertainties. 

\subsection{Measurement Update}\label{sec:MU}

Depending on the choice of a reference direction, there are two types of direction measurements. 
In the first type, the reference direction is known in the inertial frame and the measurement reading is resolved in the body-fixed frame.
For example, this corresponds to a magnetometer that measures the direction of magnetic field, or an accelerometer that measures the direction of gravity.
Both commonly appear in the Inertial Measurement Units (IMU) utilized in aircraft or robotic systems. 
Another example would be a sun sensor or a star tracker in spacecraft that determine the direction towards the sun or a selected set of distant stars. 

Alternatively, in the less common second type, the reference direction is known in the body-fixed frame, and the measurement reading is resolved in the inertial frame. 
For instance, differential GPS has been utilized in attitude determination of an aircraft, where two GPS antennas are attached to, say, the left wing-tip and the right-wing tip, and GPS measurements provide the unit-vector from the left wing to the right wing in the inertial frame.

These are referred to as \textit{inertial direction measurement} and \textit{body-fixed direction measurement}, respectively, throughout this paper. 
Analogous to the fact that uncertainties propagated through the right-trivialized \eqref{eqn:dR_w} are distinct from the left-trivialized \eqref{eqn:dR_W}, the information available from an inertial direction measurement is different from a body-fixed direction measurement, as discussed below.

First, consider inertial direction measurements. 
Let $\mathbf{a}$ be a reference vector fixed in the inertial frame, and $a\in\Sph^2 =\{q\in\Re^3\,|\, \|q\|=1\}$ be the corresponding coordinates of the vector $\mathbf{a}$ resolved in the inertial frame. 
Thus, $a$ is known and fixed.
The reference direction is measured with a sensor fixed to the body, which provides the measurement $x\in\Sph^2$ resolved in the body-fixed frame. 
In the absence of noise, the measurement should be $x = R^T a$.
Instead, it is assumed that the sensor measurement for a given attitude  is distributed according to the von-Mises Fisher distribution on $\Sph^2$~\cite{MarJup99} as follows:
\begin{align} \label{eqn:x|R}
    p(x|R) = \frac{\kappa}{4\pi\sinh \kappa} \exp( \kappa a^T Rx),
\end{align}
for a parameter $\kappa>0$ that determines the degree of dispersion. 
This distribution is centered at $R^T a$ and it is more concentrated as $\kappa$ is increased, implying a more accurate sensor. 

The additional information gathered by $x$ is summarized by the following formulation of measurement update. 
\begin{theorem} \label{thm:FI_post}
    Let $R\sim\mathcal{M}(F^-)$ be a prior distribution for a given $F^-\in\Re^{3\times 3}$, and let $a\in\Sph^2$ be the coordinates of a reference vector $\mathbf{a}$ in the inertial frame.
    The direction measurement of $\mathbf{a}$ resolved in the body-fixed frame is given by $x\in\Sph^2$. 
    Then, the posterior distribution of the attitude conditioned by $x$ is also matrix Fisher with $R|x \sim \mathcal{M}(F_I^+)$, where the posterior matrix parameter $F_I^+\in\Re^{3\times 3}$ is given by
\begin{align}
    F_I^+ = F^- + \kappa a x^T.\label{eqn:F_I}
\end{align}
\end{theorem}
\begin{proof}
    See \cite{LeeITAC18}.
\end{proof}

Similarly, the body-fixed direction measurement is considered as follows. 
Let $\mathbf{b}$ be a vector fixed to the body-fixed frame, and $b\in\Sph^2$ be its coordinates resolved in the body-fixed frame. 
Similar to the above, $b$ is known and fixed. 
The vector $\mathbf{b}$ is measured by a sensor generating a measurement $y\in\Sph^2$ resolved in the inertial frame. 
In the absence of noise, we have $y = R b$, and $y|R$ is assumed to be distributed by
\begin{align} \label{eqn:y|R}
    p(y|R) = \frac{\kappa}{4\pi\sinh \kappa} \exp( \kappa b^T R^T y),
\end{align}
for a parameter $\kappa>0$.

\begin{theorem} \label{thm:FB_post}
    Let $R\sim\mathcal{M}(F^-)$ be a prior distribution for a given $F^-\in\Re^{3\times 3}$, and let $b\in\Sph^2$ be the coordinates of a reference vector $\mathbf{b}$ in the body-fixed frame.
    The direction measurement of $\mathbf{b}$ resolved in the inertial frame is given by \\ $y\in\Sph^2$. 
    Then, the posterior distribution of the attitude conditioned by $y$ is also matrix Fisher with $R|y \sim \mathcal{M}(F_B^+)$, where the posterior matrix parameter $F_B^+\in\Re^{3\times 3}$ is given by
\begin{align}
    F_B^+ = F^- + \kappa y b^T.\label{eqn:F_B}
\end{align}
\end{theorem}
\begin{proof}
    This can be easily shown by the Bayes' rule, and the detailed procedure is similar with \cite{LeeITAC18}.
\end{proof}

\subsection{Difference Between Inertial and Body-Fixed Measurement}

Now, we study the implication of \eqref{eqn:F_I} and \eqref{eqn:F_B}.
Suppose that the attitude is completely unknown before the measurement, i.e., $F^-=0_{3\times 3}$. 
For the inertial direction measurement, the matrix parameter \eqref{eqn:F_I} for the posterior distribution is decomposed into
\begin{align} \label{eqn:SVD_FI+}
    F_I^+  = \begin{bmatrix} a & a' & a'' \end{bmatrix} 
    \mathrm{diag}[\kappa, 0,0] \begin{bmatrix} x & x' & x'' \end{bmatrix}^T,
\end{align}
where $a',a''\in\Sph^2$ are chosen such that $a' \cdot a =0$ and $a'' = a\times a'$, which ensure that the matrix $[a,a',a'']\in\SO$.
The other $x',x''\in\Sph^2$ are defined similarly. 
Therefore, the above is written in the form of proper singular value decomposition with $S=\mathrm{diag}[\kappa,0,0]$. 
The first principal axis is $a$ when resolved in the inertial frame, or $x$ when resolved in the body-fixed frame. 
Further, the rotation about the first principal axis is completely unknown as $s_2+s_3 = 0$.

In other words, the marginal distribution of each body-fixed axis will make a circle normal to the fixed $\mathbf{a}$.
These are illustrated in the left column of \Cref{fig:mea}, where three different measurements $x$ are considered for the reference direction  chosen as $a=e_1$ with the parameter $\kappa = 500$. 
These show that a single inertial direction measurement determines the attitude up to the rotation about $\mathbf{a}$.
Since $\mathbf{a}$ is fixed in the inertial frame, the direction of this ambiguity does not change. 

Next, the matrix parameter \eqref{eqn:F_B} for the posterior distribution of a body-fixed direction measurement is decomposed into
\begin{align} \label{eqn:SVD_FB+}
    F_B^+  = \begin{bmatrix} y & y' & y'' \end{bmatrix} 
    \mathrm{diag}[\kappa, 0,0] \begin{bmatrix} b & b' & b'' \end{bmatrix}^T,
\end{align}
where $y',y''\in\Sph^2$ and $b',b''\in\Sph^2$ are chosen such that the corresponding enclosing matrices belong to $\SO$, respectively. 
The resulting first principal axis is $b$ when resolved in the body-fixed frame. 

In other words, the marginal distribution of each body-fixed axis will make a circle normal to $\mathbf{b}$ fixed in the body-fixed frame.
These are illustrated in the right column of \Cref{fig:mea}, where three different measurements $y$ are considered for the reference direction chosen as $b=e_1$ with the parameter $\kappa = 500$. 
Here, the particular choice of $b=e_1$ makes the marginal distribution of the second body-fixed axis overlap with the third-body fixed axis. 
If $b$ is not aligned with any body-fixed axis of the mean attitude, it would yield three rings as in the left column of the figure. 
The important distinction from the inertial direction measurement is that these circles for the marginal distribution rotate as the body rotates. 

\begin{figure}
    \centerline{
        \subfigure[$x=e_1$]{
            \includegraphics[width=0.45\columnwidth, trim=130 80 130 60, clip]{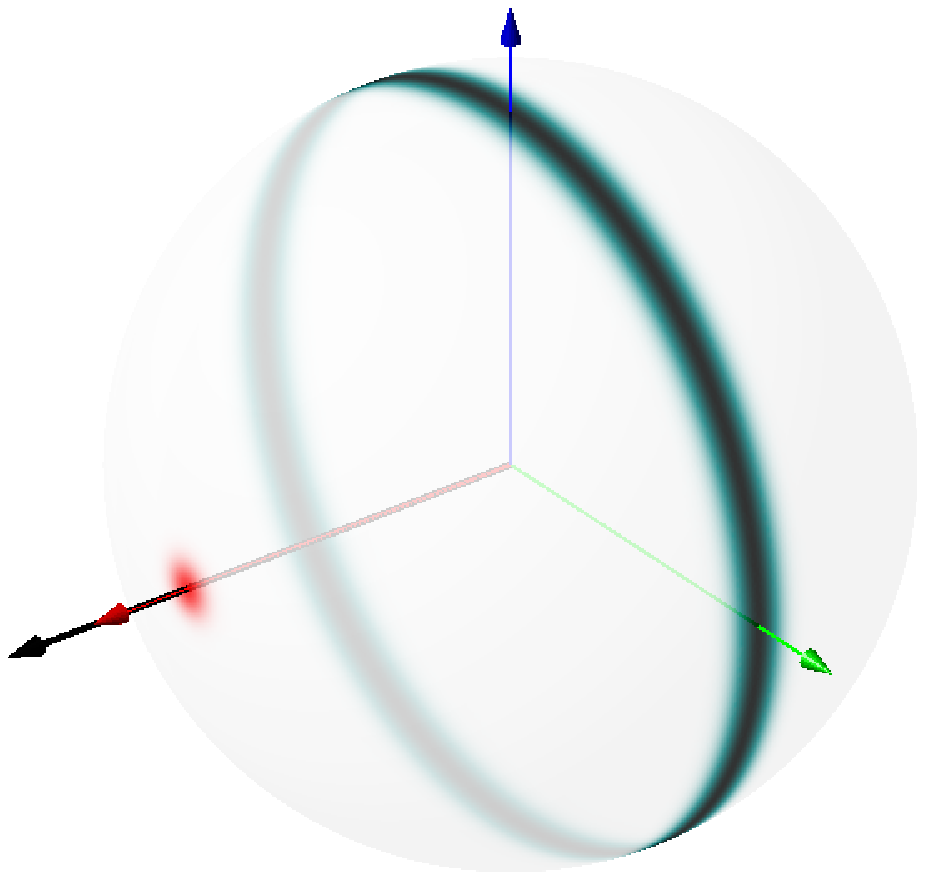}
            \begin{tikzpicture}[overlay]
                \draw[arrows={-Triangle[angle=30:4pt]}] (-2.9,0.25) -- ++(90:0.5);
                \draw[arrows={-Triangle[angle=30:4pt]}] (-2.9,0.25) -- ++(-30:0.5);
                \draw[arrows={-Triangle[angle=30:4pt]}] (-2.9,0.25) -- ++(210:0.5);
                \node at (-3.3,0.2) {\scriptsize $\mathbf{e}_1$};
                \node at (-2.5,0.2) {\scriptsize $\mathbf{e}_2$};
                \node at (-2.9,0.85) {\scriptsize $\mathbf{e}_3$};
            \end{tikzpicture}
            \label{fig:mea_I_1}
        }
        \subfigure[$y=e_1$]{
            \includegraphics[width=0.45\columnwidth, trim=150 80 120 60, clip]{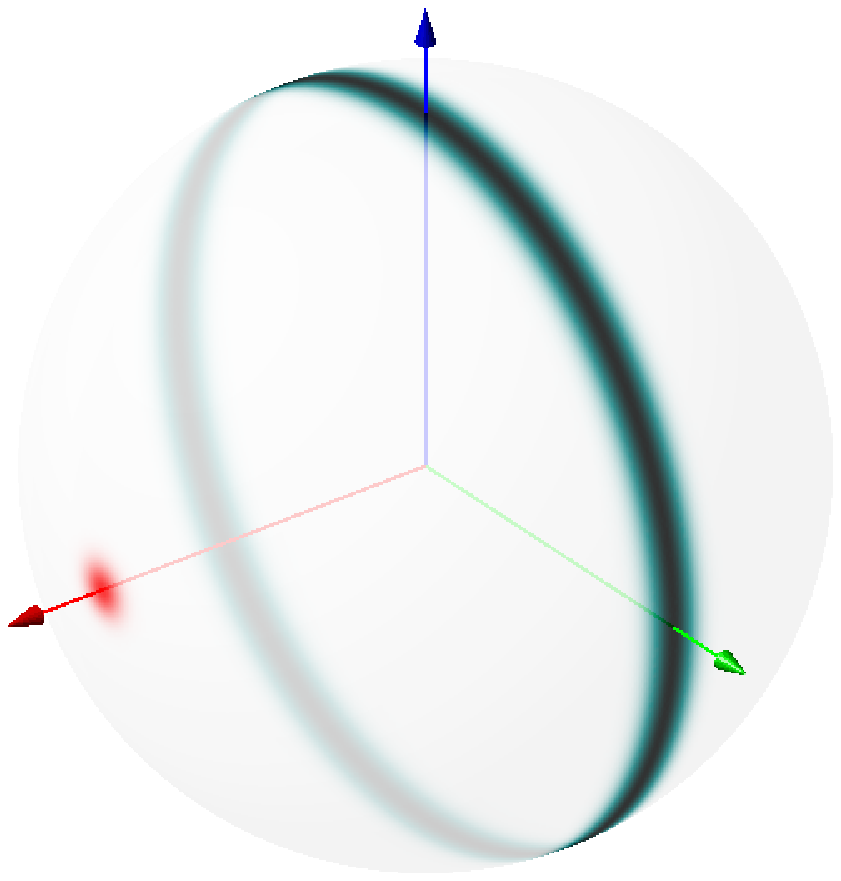}
        }
    }
    \centerline{
        \subfigure[$x=\cos\frac{\pi}{6}e_1 + \sin\frac{\pi}{6} e_2$]{
            \includegraphics[width=0.45\columnwidth, trim=130 80 130 60, clip]{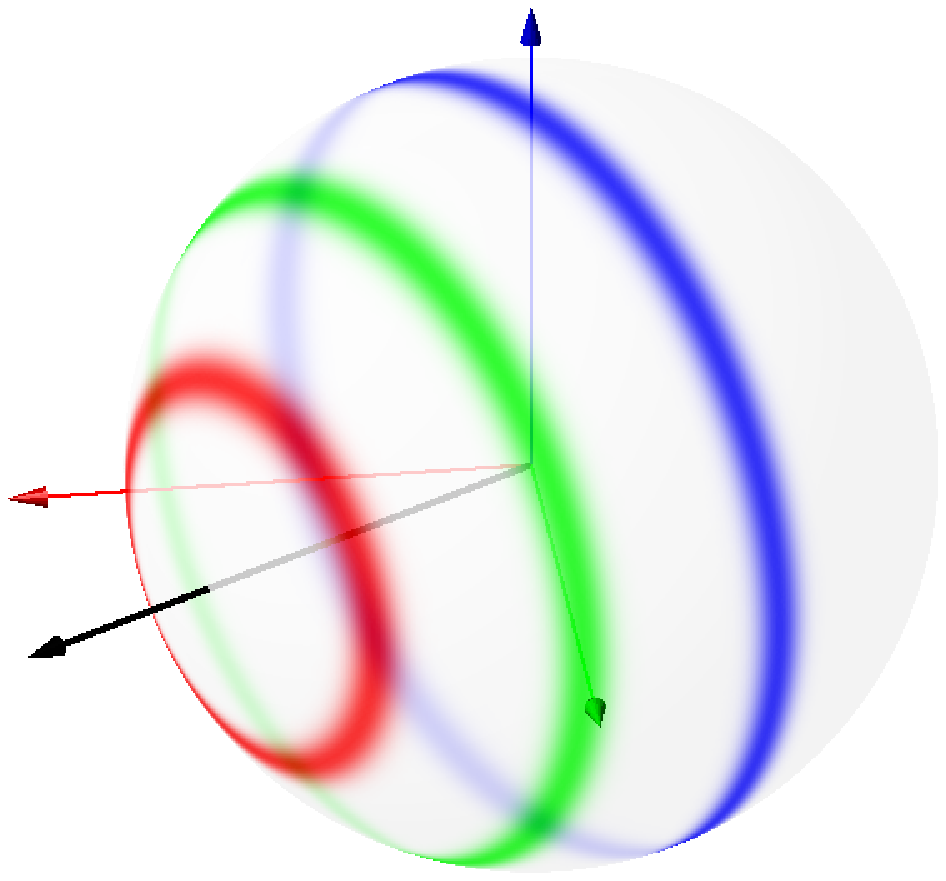}
            \label{fig:mea_I_2}
        }
        \subfigure[$y=\cos\frac{\pi}{6}e_1 + \sin\frac{\pi}{6} e_2$]{
            \includegraphics[width=0.45\columnwidth, trim=150 80 120 60, clip]{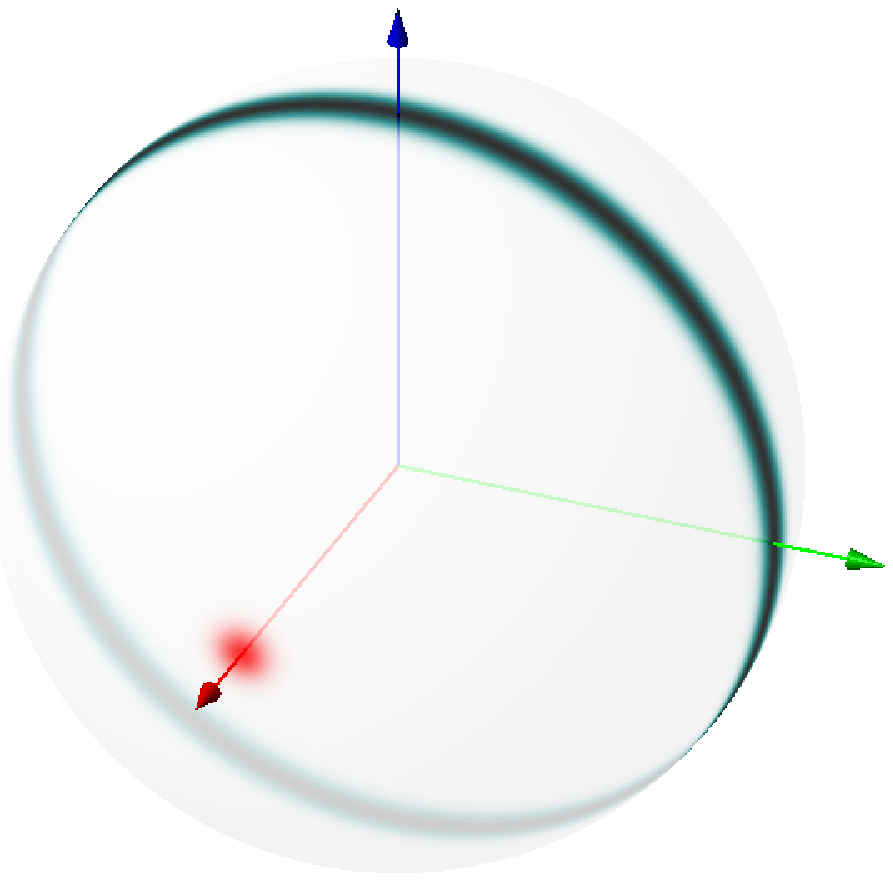}
        }
    }
    \centerline{
        \subfigure[$x=\cos\frac{2\pi}{3}e_1 + \sin\frac{2\pi}{3} e_2$]{
            \includegraphics[width=0.45\columnwidth, trim=140 80 120 60, clip]{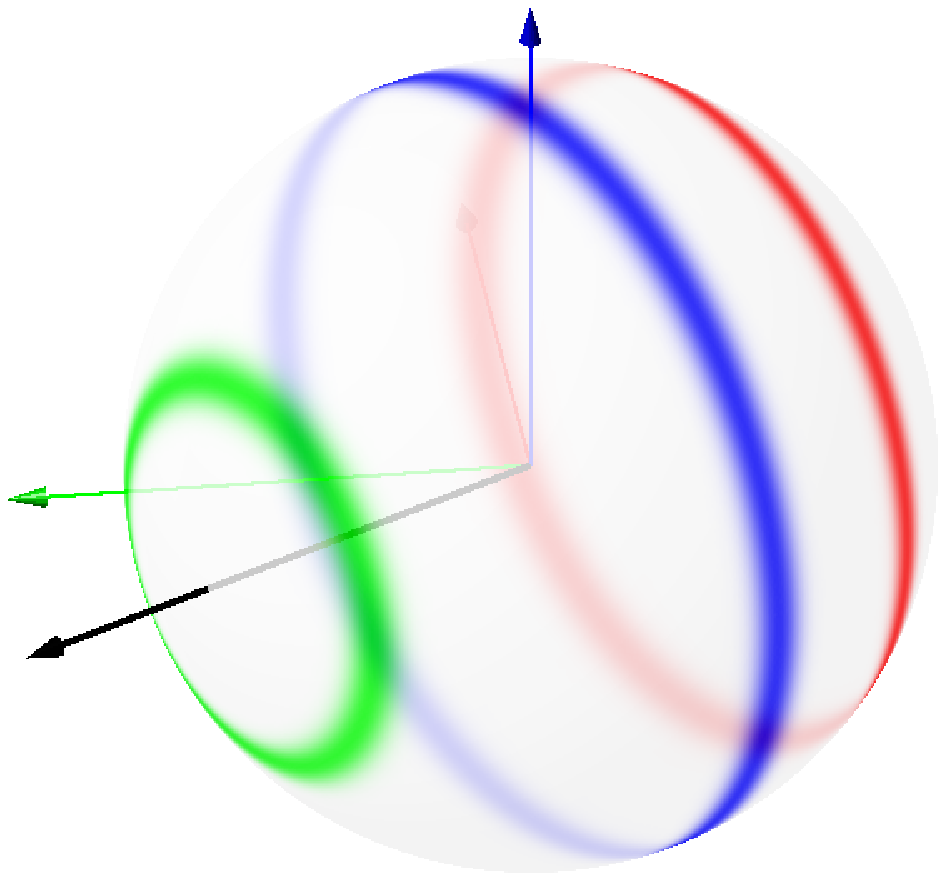}
            \label{fig:mea_I_3}
        }
        \subfigure[$y=\cos\frac{2\pi}{3}e_1 + \sin\frac{2\pi}{3} e_2$]{
            \includegraphics[width=0.45\columnwidth, trim=140 80 130 60, clip]{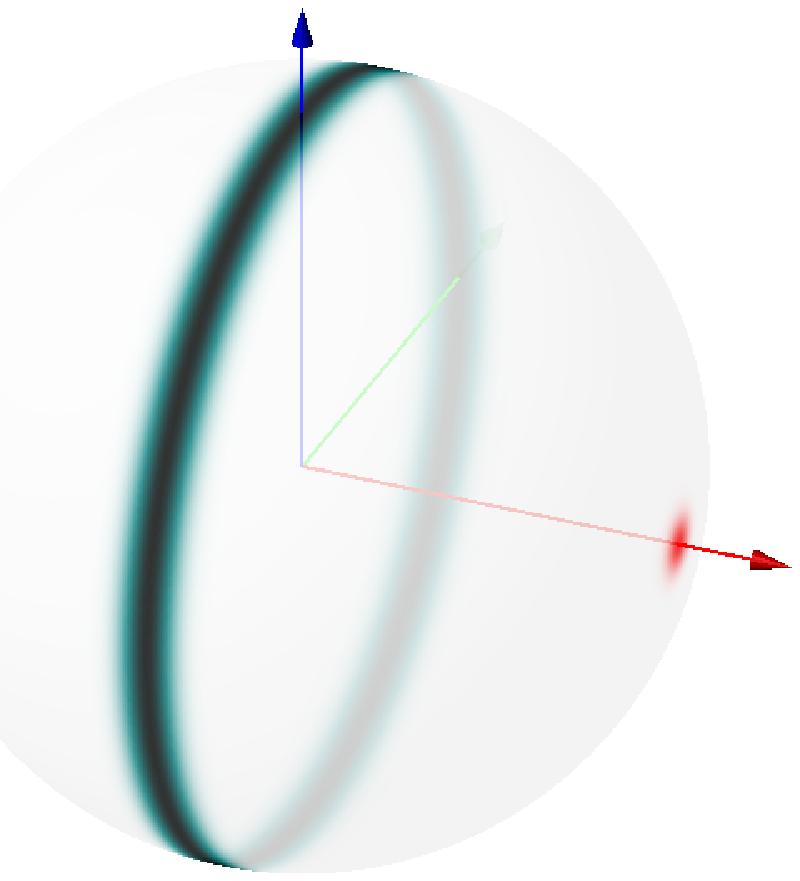}
        }
    }
    \caption{Posterior distribution with several single direction measurements: (a,c,e) inertial direction measurements with $a=e_1$; (b,d,f) body-fixed direction measurements with $b=e_1$.}\label{fig:mea}
\end{figure}

\section{Stochastic Attitude Observability} \label{sec:SAO}

Before proceeding to show the result in \Cref{table:observability}, we first define formally what the observability of attitude is in the stochastic sense.
In estimation of a stochastic dynamical system, the posterior distribution of the state conditioned by measurements carries the complete information we know about the state.
Therefore, the question is to what extent we may infer the state from its posterior distribution.
In this section, we address this problem through two aspects. 
First, given the posterior distribution on $\SO$, we want to know if the attitude that minimizes the mean squared error is unique, which indicates whether a single ``best'' estimate can be deduced from the distribution.
Second, we want to study the positive definiteness for the Fisher information of the mean attitude for the posterior distribution, which implicates whether the information contained in the distribution is enough to decide the mean attitude.
These two aspects are studied in details in the following two subsections.

\subsection{Uniqueness of Attitude Estimate}

A popular method to infer an attitude from a density function or some weighted random samples on $\SO$ is to solve an optimization problem that minimizes the mean squared Frobenius norm~\cite{oshman2004estimating,MarPMDSAS05}, defined as follows.
\begin{definition}
    Let $p(R)$ be the probability density function for a random $R\in\SO$. 
    Its minimum mean square estimate (MMSE) is defined as
    \begin{align}
        \mathrm{M}_{\mathrm{MMSE}} [R] = \argmin_{Q\in\SO} \{\E[\| R - Q\|_F^2 ]\},
    \end{align}
    where $\|\cdot \|_F$ denotes the Frobenius norm. 
    Or, equivalently, 
    \begin{align}
        \mathrm{M}_{\mathrm{MMSE}} [R] = \argmax_{R\in\SO} \{\trs{R^T \E[R]}\}. \label{eqn:MMSE_max}
    \end{align}
\end{definition}

Also, \eqref{eqn:MMSE_max} is closely related to Wahba's problem~\cite{WahSR65}, whose solution is well established in terms of the proper singular value decomposition~\cite{MarJAS88}.
Here, we present the solution of the above MMSE with an additional careful analysis on its uniqueness.

\begin{lem}\label{lem:MMSE}
    Suppose $R\in\SO$ is a random rotation matrix.
    Let the proper singular value decomposition of the first moment be $\E[R]=UDV^T$ where $U,V\in\SO$ and $D=\mathrm{diag}[d_1,d_2,d_3]\in\Re^{3\times 3}$ for $d_1\geq d_2\geq |d_3|\geq 0$. 
    Depending on $D$, the MMSE of $R$ is given by
    \begin{enumerate}
        \item $d_2+d_3 >0$: $UV^T$ (unique),
        \item $d_1\neq d_2$ and $d_2+d_3=0$: $\{U\exp(\theta\hat e_1) V^T\,|\,\theta\in[-\pi,\pi)\}$ (1D),
        \item $d_1=d_2 = -d_3 >0$: $\{U\exp(\theta\hat a) V^T|\, a\in\Sph^2,\; a_3 =0, \theta\in[-\pi,\pi)\}$ (2D),
        \item $d_1=d_2=d_3=0$:  $\SO$ (3D),
    \end{enumerate}
    where the number in the parentheses indicates the dimension of the set corresponding to the solution of MMSE.
\end{lem}
\begin{proof}
    Let $\theta\in[-\pi,\pi)$ and $a\in\Sph^2$ be defined such that $U^T R V = \exp(\theta\hat a)$. 
    Then, we have
    \begin{align}
        \trs{R^T \E[R]} 
                         & = \trs{D} - (1-\cos\theta) \sum_{i=1}^3 (d_i+d_j) a_k^2, \label{eqn:trsRER}
    \end{align}
    where $(i,j,k)\in\{(1,2,3),(2,3,1),(3,1,2)\}$.
    Since $d_1+d_2 \geq d_3+d_1 \geq d_2 +d_3 \geq 0$, the above is maximized when $\theta=0$, or equivalently $R= UV^T$. 

    The uniqueness is contributed by two aspects: the uniqueness of $UV^T$ from the proper singular value decomposition and the uniqueness of the maximum at \eqref{eqn:trsRER}.
    First note that the non-uniqueness of proper singular vectors caused by any simultaneous sign change of the corresponding columns of $U$ and $V$ does not affect the uniqueness of $UV^T$, therefore we will only consider the non-uniqueness of $U,V$ caused by repeated singular values.
    \begin{enumerate}
        \item $d_2+d_3 >0$:
            If $d_3\geq 0$, the polar decomposition of $\E[R]$ is written as $\E[R] = (UV^T)(VDV^T)$, where $VDV^T = (\mathrm{E}[R]^T\mathrm{E}[R])^{1/2}$ is uniquely determined.
            The rank of $VDV^T$ is at least two, so $UV^T$ rotates two independent columns of $VDV^T$ to the corresponding columns of $\E[R]$, and therefore it is unique.

            Next, when $d_3<0$, consider two sub-cases: (i) if $d_1\neq d_2$, then both $U$ and $V$ are unique; (ii) if $d_1=d_2$, $(U,V)$ can be replaced by $(U\exp(\phi\hat{e}_3),V\exp(\phi\hat{e}_3))$ for any $\phi\in[-\pi,\pi)$, but $UV^T$ is still unique.
        \item[2a)] $d_1 \neq d_2 =  d_3 =  0$:  We have
            \begin{align*}
                \trs{R^T \E[R]} & = \trs{D} - (1-\cos\theta) d_1 (1-a_1^2),
            \end{align*}
            which is maximized for any $R$ in the given set of $a=e_1$. 
            The matrices $U$ and $V$ can be replaced with 
            $U\exp(\phi_1\hat e_1)$ and $V\exp(\phi_2\hat e_1)$, respectively, for any $\phi_1,\phi_2\in[-\pi,\pi)$. 
            However, the ambiguity of $U,V$ does not enlarge the set of $R$ maximizing \eqref{eqn:trsRER}.
        \item[2b)] $d_1 \neq d_2 = -d_3 >0$: Similarly, \eqref{eqn:trsRER} is maximized for any $R$ in the given set.
        The matrices $(U,V)$ can be replaced with 
        $(U\exp(\phi\hat e_1), V\exp(-\phi\hat e_1))$ for any $\phi\in[-\pi,\pi)$.
        But, it does not alter the given set. 
        \item[3)] $d_1 = d_2 = -d_3 >0$: We have
        	\begin{align*}
        		\trs{R^T \E[R]} = \trs{D} - (1-\cos\theta)(d_1+d_2)a_3^2,
        	\end{align*}
        	which is maximized when $a_3=0$. 
            The ambiguity of $U$ and $V$ is written as $U\exp(\phi_1\hat e_1+\phi_2 \hat e_2 +\phi_3 \hat e_3)$ and $V\exp(-\phi_1\hat e_1 - \phi_2\hat e_2 +\phi_3\hat e_3)$, respectively for any $\phi_1,\phi_2,\phi_3\in[-\pi,\pi)$.
            Same as above, this does not alter the set of $R$ maximizing \eqref{eqn:trsRER}.
        \item[4)] $d_1 = d_2=d_3=0$: This is a trivial case when $\trs{R^T \E[R]}=0$ for any $R\in\SO$.
    \end{enumerate}
    These complete the proof. 
\end{proof}

For all cases, the MMSE contains $UV^T$.
Nevertheless, only when $d_2+d_3>0$, the MMSE is unique;
otherwise it can only be determined up to a rotation, where the dimension of the set representing the solution of MMSE is equivalent to 3 minus the rank of $\tr{D}I_{3\times 3}-D = \mathrm{diag}[d_2+d_3, d_1+d_3, d_1+d_2]$.
Therefore, we claim that the attitude is completely observable given a density function on $\SO$ if $\tr{D}I_{3\times 3}-D$ is positive-definite, i.e. when the MMSE is unique.

Note that the MMSE and the observability are determined solely by $\E[R]$ without assuming any particular form of $p(R)$. 
However, if the attitude is assumed to follow a matrix Fisher distribution, since $s_i+s_j$ shares the same sign with $d_i+d_j$ as indicated in \Cref{lemma:SD}, the above condition is equivalent to that $\tr{S}I_{3\times 3}-S$ is positive-definite.

\subsection{Information Theoretic Observability Analysis}

Next, we derive another attitude observability criterion from an information theoretic perspective.
Suppose $R$ follows a matrix Fisher distribution $\mathcal{M}(F)$ with the proper singular value decomposition of $F$ given by $F=USV^T$.
Then the MMSE of $R$ is the mean attitude $R^*=UV^T$.
In this subsection, we calculate the Fisher information $\mathcal{I}(R^*)$ of the mean attitude $R^*$ to study the observability of $R$.

We first study a more general problem of estimating $U$, $S$, and $V$ of the matrix Fisher distribution. 
Now, we have a family of probability density functions $p(R|U,S,V)$ parameterized by $(U,S,V)\in\SO\times\Re^3\times \SO$, which we wish to estimate from $R$. 
The resulting log likelihood is
\begin{align}
    l(R|U,S,V) = \trs{F^T R} - \log c(S).\label{eqn:l}
\end{align}
The Fisher information captures the variability of the derivatives of the log likelihood with respect to the parameters, or roughly speaking, the sharpness of the log likelihood~\cite{Chi12v1}.
It is considered as a measure of information that the random variable carries about the unknown parameters, as a smaller Fisher information implies that the likelihood does not vary much with respect to the parameters, thereby making it hard to estimate them. 
The inverse of the Fisher information also serves as a lower bound of the variance of the estimated parameter, according to the Cram\'{e}r--Rao inequality.

Here the conventional Fisher information matrix cannot be directly applied as the unknown parameters reside in a non-Euclidean manifold. 
Instead, the method in \cite{smith2005covariance} which generalizes the Fisher information on Riemannian manifolds is used.
The key idea is considering the Fisher information as a metric, or more explicitly, $\E[dl\otimes dl]$, which is identical to $\E[-\nabla^2 l]$~\cite[Theorem 1]{smith2005covariance}.~\footnote{More rigorous definition of a metric, a tensor product $\otimes$, or a covariant derivative $\nabla$ is beyond the scope of this paper, and it is relegated to other monographs in differential geometry, such as~\cite{MarRat99}.}
The corresponding Fisher information matrix of \eqref{eqn:l} is constructed as follows. 

\begin{lem}\label{lem:FIM}
    The Fisher information matrix of \eqref{eqn:l}, namely $\mathcal{I}(U,S,V):\Re^9\times\Re^9\rightarrow \Re$ is constructed as 
    \begin{align}
    &\mathcal{I}(U,S,V) = -\E[ \nabla^2 l (R|U,S,V) ]\nonumber \\
    &=\begin{bmatrix}
        \trs{DS}I_{3\times 3}- DS  & 0 & \sum_{i=1}^3 e_i^T s \hat e_i D\hat e_i \\
        0 & \frac{\partial^2 \log c(S)}{\partial s^2} & 0 \\
        \sum_{i=1}^3 e_i^T s \hat e_i D\hat e_i& 0 & \trs{DS}I_{3\times 3} -DS
    \end{bmatrix},\label{eqn:FIM}
    \end{align}
    where $D\in\Re^{3\times 3}$ is the diagonal matrix composed of the proper singular values of $\E[R|U,S,V]$.
\end{lem}
\begin{proof}
    Let $\Q = \SO\times\Re^3\times\SO $, and $q=(U,S,V)\in\Q$.
    The tangent space $\T_q\Q$ is identified with $\T_q\Q\simeq \Re^9$ through the hat map, and the cotangent space is also identified with $\Re^9$ using the dot product.
    More specifically, for $\xi=(u,\varsigma,v)\in\Re^9$, the corresponding tangent vector is given by $(U\hat u, \varsigma, V\hat v)\in\T_q \Q$. 

    Since $l$ is a real-valued function on $\Q$, its covariant derivative $\nabla_\xi l$ along $\xi$ is equivalent to the differential $dl(\xi)$ given by
    \begin{align*}
        \nabla_\xi l & = \frac{d}{d\epsilon}\bigg|_{\epsilon =0} l( R | U\exp(\epsilon\hat u), S+\epsilon\mathrm{diag}[\varsigma], V\exp(\epsilon\hat v))\\
                     & = \trs{(U\hat uSV^T + U\mathrm{diag}[\varsigma]V^T - US\hat v V^T)^TR}\\
                     & \quad - \deriv{\log c(S)}{s}\cdot \varsigma \\
                     & = 
                     \begin{bmatrix}
                         (QS-SQ^T)^\vee \\ 
                         \mathrm{diag}[Q]-\frac{1}{c(S)}\deriv{c(S)}{s} \\
                         (Q^TS - SQ)^\vee
                     \end{bmatrix}\cdot \xi,
    \end{align*}
    where $Q=U^TRV$.
    Because $\E[Q] = D$ is diagonal, it is straightforward to show $\E[dl(\xi)]=0$ for any $\xi$.
    
    The second order covariant derivative of $l$ along $\xi_1$ and $\xi_2$ is given by $\nabla^2_{\xi_1,\xi_2} l = \xi_2(\xi_1 l) - (\nabla_{\xi_2}\xi_1) l = \xi_2 ( dl(\xi_1)) - dl(\nabla_{\xi_2} \xi_1)$,
    where the second term vanishes after taking expectation.
    The first term is bi-linear in $\xi_1$ and $\xi_2$, thus it can be written as a matrix as in \eqref{eqn:FIM}.
    Suppose $\xi_1 = (u_1,0,0)$ and $\xi_2 =(u_2,0,0)$.
    We have
    \begin{align*}
        \xi_2 (dl(\xi_1))  & = (-\hat u_2 QS -SQ^T \hat u_2)^\vee \cdot u_1\\
                           &= u_1^T \braces{\frac{1}{2}(QS+SQ^T) - \tr{QS}I_{3\times 3}} u_2.
    \end{align*}
    Taking the expectation of the expression in the braces with $\E[Q]=D$ and multiplying it with $-1$ yield the upper-left 3-by-3 block of \eqref{eqn:FIM}.
    The remaining blocks can be obtained similarly.
\end{proof}

According to~\cite[Theorem 2]{smith2005covariance}, the inverse of \eqref{eqn:FIM} serves as the lower bound of the covariance of the estimated parameters for any unbiased estimator, after neglecting the curvature terms at small errors and biases. 
Recall that the objective is to estimate the attitude $R^*=UV^T$, whose perturbation is 
\begin{align*}
    \delta R^* = U\hat u V^T - U\hat vV^T.
\end{align*}
Let $\eta = u-v\in\Re^3$ so that $\delta R^* = U\hat\eta V^T$. 
Similar with \eqref{eqn:dij_alpha}, the Fisher information matrix to estimate $R^*$ is obtained by left-multiplying \eqref{eqn:FIM} with $\frac{1}{2} [I_{3\times 3}; 0_{3\times 3}; -I_{3\times 3}]$ and by right-multiplying with its transpose to obtain
\begin{align}
    \mathcal{I}(UV^T) = \frac{1}{2} \mathrm{diag} 
    \begin{bmatrix} (d_2+d_3)(s_2+s_3) \\ (d_3+d_1)(s_3+s_1) \\ (d_1+d_2)(s_1+s_2) \end{bmatrix}.\label{eqn:FIM_eta}
\end{align}
Therefore, for any unbiased estimator of $R^*$, its error covariance will not be smaller than the inverse of the above, up to additional curvature terms, and we can take the determinant or the minimum eigenvalue of \eqref{eqn:FIM_eta} as a measure of observability. 

In summary, we have presented two approaches for stochastic attitude observability: \Cref{lem:MMSE} is based on the uniqueness of MMSE for an arbitrary density, and \Cref{lem:FIM} relies on the particular matrix Fisher distribution for an arbitrary unbiased estimator. 
By \Cref{lemma:SD}, $s_i+s_j$ and $d_i+d_j$ share the same sign, hence both approaches provide consistent results.

Based on these we formulate stochastic attitude observability for an arbitrary density as follows.
\begin{definition} \label{def:observability}
    A random rotation matrix $R\sim p(R)$ is stochastically observable if $d_2+ d_3 >0$, or equivalently
    \begin{align}
        \mathcal{O} = \trs{D}I_{3\times 3} - D \succ 0,\label{eqn:OC}
    \end{align}
    where $D=\mathrm{diag}[d_1,d_2,d_3]$ is the proper singular value of $\E[R]$.  
    The corresponding measure of observability is
    \begin{align}
        \rho(R) = \det[\mathcal{O}] = (d_1+d_2)(d_3+d_1)(d_2+d_3).
    \end{align}
\end{definition}
It is straightforward to show that \eqref{eqn:OC} is equivalent to $\trs{S}I_{3\times 3} - S\succ 0$ when $R\sim\mathcal{M}(USV^T)$.

\section{Attitude Observability with Single Direction Measurements}

We have presented techniques to propagate attitude uncertainties in \Cref{sec:UP}, and to update attitude uncertainties with single direction measurements in \Cref{sec:SDM}.
Bayesian attitude estimation is to construct the probability distribution of the attitude conditioned by the history of measurements for a given initial distribution, 
and it is readily addressed by composing the uncertainty propagation scheme in \Cref{sec:UP} with the measurement update in \Cref{sec:SDM}:
the initial distribution is propagated until the first measurement becomes available to correct the distribution, which is propagated until the next measurement, and these are repeated. 
The posterior distribution can then be tested for observability using the criterion presented in Section \ref{sec:SAO}.

As there are two cases for each of uncertainty propagation and measurement update, we have four possible combinations to estimate attitude, as summarized in \Cref{table:observability}.
This section presents two combinations that yield unobservability, and two other cases resulting in attitude observability with single direction measurements. 
The same results can also be derived in deterministic sense as presented in Appendix \ref{app:deterministic}.

\subsection{Combinations with Unobservability}

We first present why the common IMU cannot estimate the full attitude with single direction measurements.
In a typical IMU, the angular velocity is measured in the body-fixed frame using a gyroscope, and the reference direction in the inertial frame, such as the direction of gravity or magnetic field, is measured in the body-fixed frame. 
As such it is a combination of the left-trivialized \eqref{eqn:dR_W} and the inertial direction measurement \eqref{eqn:F_I}.
Looking at the right column of \Cref{fig:prop} and the left column of \Cref{fig:mea}, it is clear why the attitude cannot be determined in this case: the direction about which the rotation is most uncertain, namely the first principal axis, remains unchanged in the inertial frame for both \eqref{eqn:dR_W} and \eqref{eqn:F_I}.

The above intuition is formulated formally in the next theorem.
\begin{theorem} \label{thm:non-estimable}
	Consider the two Bayesian attitude filters composed of
    \begin{itemize}
        \item right-trivialized angular velocity in the inertial frame  \eqref{eqn:F_R_kp} and body-fixed direction measurement \eqref{eqn:F_B}  
        \item left-trivialized angular velocity in the body-fixed frame \eqref{eqn:F_L_kp}, and inertial direction measurement \eqref{eqn:F_I} 
    \end{itemize}
    with the initial distribution $F_0=0_{3\times 3}$.
    For both cases, the attitude is not observable.
\end{theorem}
\begin{proof}
	First consider the filter given by \eqref{eqn:F_R_kp} and \eqref{eqn:F_B}.
	The propagated uncertainty before the first measurement is $F_1^- = 0_{3\times 3}$, thus $F_1 = \kappa y_1b^T$ after conditioning the first measurement.
	As shown in \eqref{eqn:SVD_FB+}, $(s_2)_1=(s_3)_1=0$, and $V_1e_1 = b$.
	We proceed with induction.
	Suppose $(s_2)_k=(s_3)_k=0$, and $V_ke_1 = b$.
	Then by \eqref{eqn:UV_R_kp}, \eqref{eqn:S_R_kp} and \Cref{lemma:SD}, the propagated parameters before the next measurement still satisfy $(s_2)_{k+1}^-=(s_3)_{k+1}^- = 0$ and $V_{k+1}^-e_1=b$.
	Next consider the update $F_{k+1} = F_{k+1}^-+\kappa y_{k+1}b^T$, which can be written as
	\begin{align*}
		F_{k+1} &= (s_1)_{k+1}^- U_{k+1}^-e_1b^T + \kappa y_{k+1}b^T \\
		&= \left((s_1)_{k+1}^-U_{k+1}^-e_1 + \kappa y_{k+1}\right)b^T \triangleq ub^T,
	\end{align*}
	where $u\in\Re^3$.
	Let $U_{k+1} = \begin{bmatrix}\frac{u}{\norm{u}}&u'&u''\end{bmatrix}$, where $u',u''\in\Sph^2$ are arbitrarily chosen such that $U_{k+1}\in\SO$.
	Also let $S_{k+1}=\mathrm{diag}[\norm{u},0,0]$, and $V_{k+1} = \begin{bmatrix}b&b'&b''\end{bmatrix}$ as in \eqref{eqn:SVD_FB+}.
    Then $F_{k+1} = U_{k+1}S_{k+1}V_{k+1}^T = ub^T $ is the proper singular value decomposition of $F_{k+1}$, and we have shown that $(s_2)_{k+1}=(s_3)_{k+1}=0$ and $V_{k+1}e_1=b$.
	Therefore, $(s_2)_k=(s_3)_k=0$ for all $k\in\mathbb{N}$, and by \Cref{lemma:SD}, the attitude is not observable.
	The proof for the second case is similar.
\end{proof}

\subsection{Combinations with Observability}

Next, consider the combination of the left-trivialized \eqref{eqn:dR_W} and the body-fixed direction measurement \eqref{eqn:F_B}.
We consider a specific case where the true attitude evolves according to
\begin{align}
    R(t) = R_0 \exp(h\hat\Omega t) = \exp(\hat\omega t) R_0,\label{eqn:R_est}
\end{align}
with $R_0=I_{3\times 3}$ and $\Omega = \omega = -\frac{\pi}{2\sqrt{3}}[1, 1, 1]\in\Re^3$, which represents a rotation about a fixed axis. 
Initially, it is assumed that the attitude is completely unknown, i.e., $F_0 = 0_{3\times 3}$, as illustrated in \Cref{fig:est_LB_F1_prior}.
At $t=0$, a body-fixed direction measurement for a reference direction $b=e_1$, i.e., the first body-fixed axis, becomes available, and $F_0$ is updated by \eqref{eqn:F_B}. 
The resulting posterior distribution at $t=0$ is shown in \Cref{fig:est_LB_F1}, where the marginal distribution of each axis forms a circle normal to the first body-fixed axis that is along $\mathbf{e}_1$.
This is propagated through $t=1$ as presented in \Cref{fig:est_LB_F3,fig:est_LB_F5}.
While the true attitude rotates, the first principal axis of the attitude estimate remains at $\mathbf{e}_1$ according to \Cref{thm:ER_L}, similar to the right column of \Cref{fig:prop}.
Next, a new body-fixed direction measurement becomes available at $t=1$. 
As the reference direction $b$ (the red axis corresponding to the first body-fixed axis) has been rotated from its initial direction $\mathbf{e}_1$, it can resolve the one-dimensional ambiguity in the attitude estimate. 
The additional information provided by the new measurement, namely $\kappa y b^T$, is overlapped with the propagated prior distribution in \Cref{fig:est_LB_FN_mea}.
The key observation is that the circles constructed by the measurement are normal to the first body-fixed axis, and they are not parallel to the propagated distribution. 
Finally, the resulting posterior distribution integrating the propagated distribution with the measurement is presented in \Cref{fig:est_LB_FN_post}, where the ambiguity is removed to provide an estimate of the complete attitude. 
In short, this illustrates that the complete attitude can be estimated by two body-fixed single direction measurements over the attitude uncertainty distribution propagated by the left-trivialized \eqref{eqn:dR_W}.
This requires that the attitude is rotated such that the reference direction $b$ is rotated between two measurements. 
For example, if the angular velocity is $\Omega=e_1$, the second measurement cannot resolve the ambiguity. 

Next, we consider the last combination of the right-trivialized \eqref{eqn:dR_w} and the inertial direction measurement \eqref{eqn:F_I}. 
Same as above, the rotation motion is chosen as \eqref{eqn:R_est} and the attitude is completely unknown initially (\Cref{fig:est_RI_F1_prior}). 
The reference direction is chosen as $a=e_1$, i.e., the first inertial axis, and the resulting posterior distribution is presented in \Cref{fig:est_RI_F1},
which is propagated according to \Cref{thm:ER_R} through $t=1$.
As shown in \Cref{fig:est_RI_F3,fig:est_RI_F5}, the propagation over \eqref{eqn:dR_w} purely rotates the distribution without altering the shape. 
At $t=1$, another measurement becomes available, and the resulting information provided by the new measurement, namely  $\kappa ax^T$ is overlapped with the propagated distribution in \Cref{fig:est_RI_FN_mea}.
\Cref{fig:est_RI_FN_mea} looks to be a combination of \Cref{fig:est_LB_F5} and \Cref{fig:est_RI_F5}, and it is further identical to \Cref{fig:est_LB_FN_mea}.
However, they are contributed from the opposite sources: in \Cref{fig:est_RI_FN_mea}, the distribution normal to the first (red) body-fixed axis are the propagated one, and the other is from the measurement; but these are reversed in \Cref{fig:est_LB_FN_mea}.
The distribution from the new measurement forms circles normal to the inertial axis $\mathbf{e}_1$, and therefore, it resolves the ambiguity of the propagated distribution corresponding to the rotation about the first body-fixed axis. 
As such, the complete attitude can be estimated by integrating those as shown in \Cref{fig:est_RI_FN_post}.

\begin{figure}
    \centerline{
        \subfigure[Prior dist. at $t=0$]{
            \includegraphics[width=0.45\columnwidth, trim=150 80 120 60, clip]{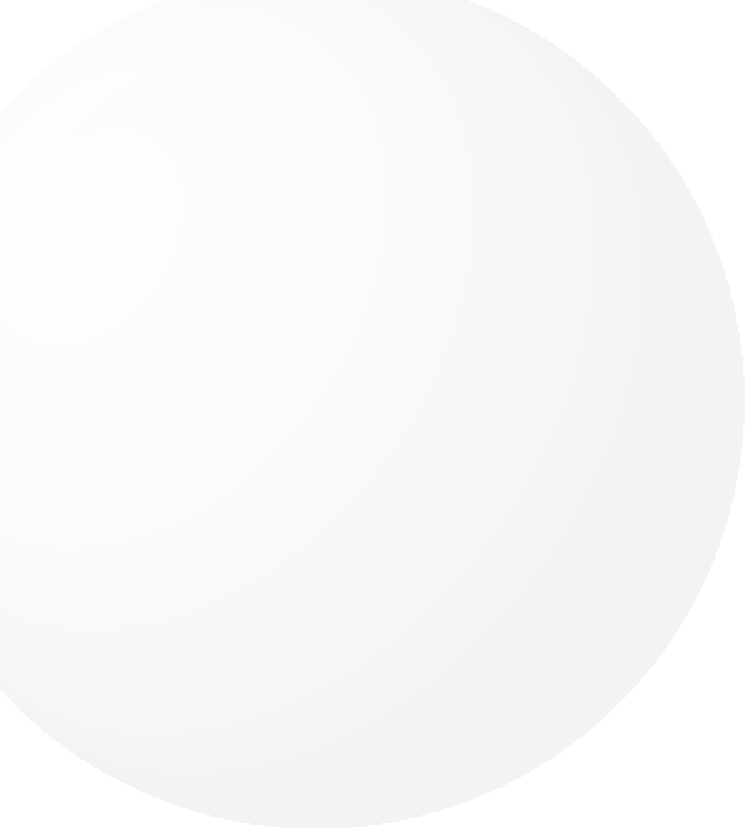}
            \begin{tikzpicture}[overlay]
                \draw[arrows={-Triangle[angle=30:4pt]}] (-2.9,0.25) -- ++(90:0.5);
                \draw[arrows={-Triangle[angle=30:4pt]}] (-2.9,0.25) -- ++(-30:0.5);
                \draw[arrows={-Triangle[angle=30:4pt]}] (-2.9,0.25) -- ++(210:0.5);
                \node at (-3.3,0.2) {\scriptsize $\mathbf{e}_1$};
                \node at (-2.5,0.2) {\scriptsize $\mathbf{e}_2$};
                \node at (-2.9,0.85) {\scriptsize $\mathbf{e}_3$};
            \end{tikzpicture}
            \label{fig:est_LB_F1_prior}
        }
        \subfigure[Posterior dist. at $t=0$]{
            \includegraphics[width=0.45\columnwidth, trim=150 80 120 60, clip]{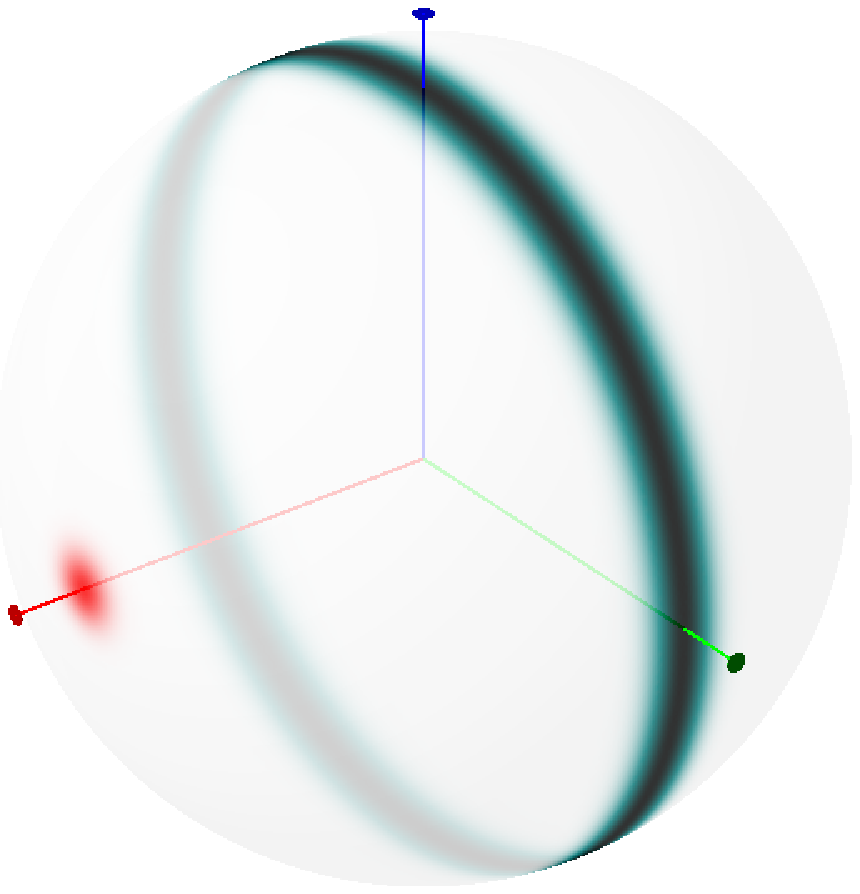}
            \label{fig:est_LB_F1}
        }
    }
    \centerline{
        \subfigure[Propagated dist. at $t=0.5$]{
            \includegraphics[width=0.45\columnwidth, trim=150 80 120 60, clip]{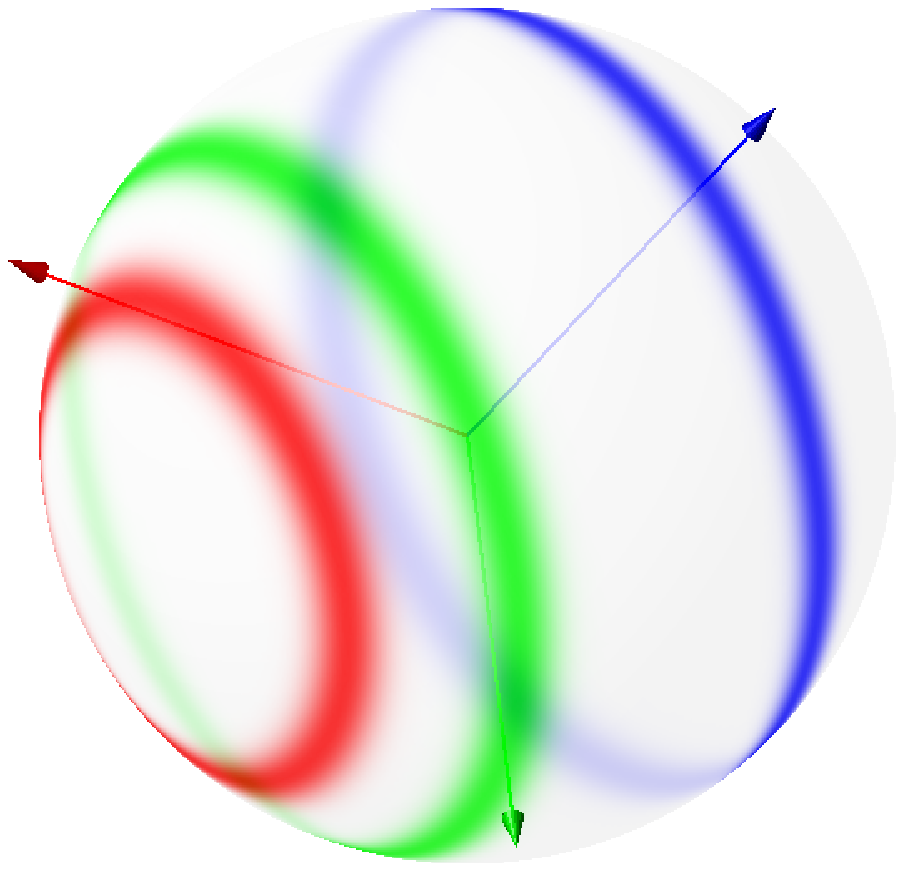}
            \label{fig:est_LB_F3}
        }
        \subfigure[Propagated dist. at $t=1$]{
            \includegraphics[width=0.45\columnwidth, trim=150 80 120 60, clip]{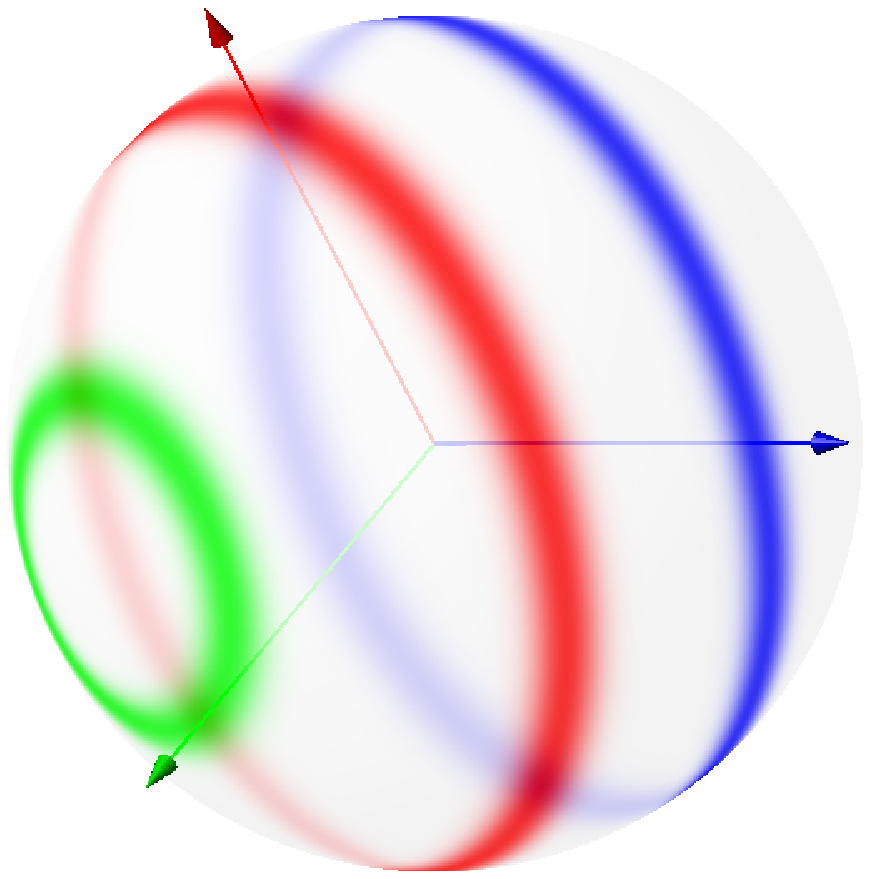}
            \label{fig:est_LB_F5}
        }
    }
    \centerline{
        \subfigure[Propagated dist. overlapped with measured dist. at $t=1$]{
            \includegraphics[width=0.45\columnwidth, trim=150 80 120 60, clip]{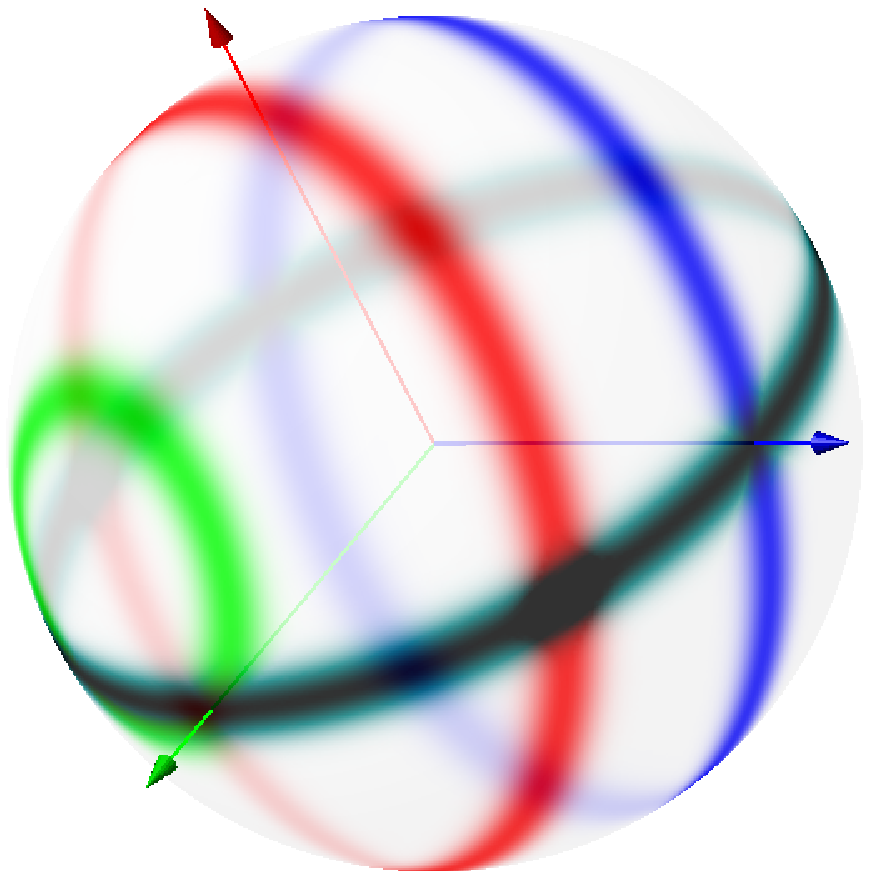}
            \label{fig:est_LB_FN_mea}
        }
        \subfigure[Posterior dist. at $t=1$]{
            \includegraphics[width=0.45\columnwidth, trim=150 80 120 60, clip]{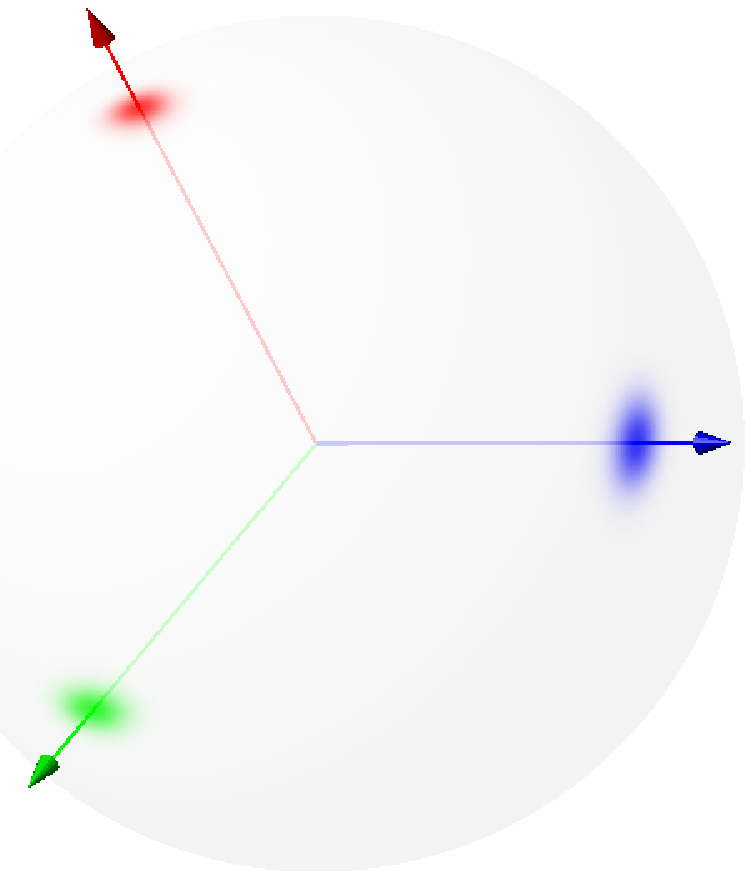}
            \label{fig:est_LB_FN_post}
        }
    }
    \caption{Estimation with the left trivialized \eqref{eqn:dR_W} and the body-fixed direction measurement \eqref{eqn:F_B}, where the complete attitude is estimated after incorporating two body-fixed single direction measurements.}\label{fig:est_LB}
\end{figure}

\begin{figure}
    \centerline{
        \subfigure[Prior dist. at $t=0$]{
            \includegraphics[width=0.45\columnwidth, trim=150 80 120 60, clip]{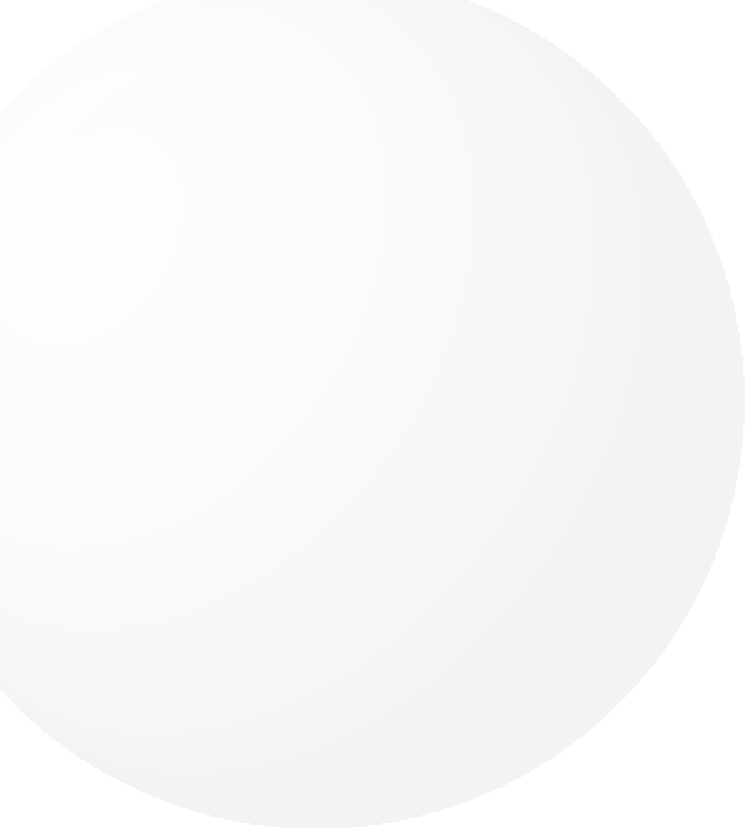}
            \begin{tikzpicture}[overlay]
                \draw[arrows={-Triangle[angle=30:4pt]}] (-2.9,0.25) -- ++(90:0.5);
                \draw[arrows={-Triangle[angle=30:4pt]}] (-2.9,0.25) -- ++(-30:0.5);
                \draw[arrows={-Triangle[angle=30:4pt]}] (-2.9,0.25) -- ++(210:0.5);
                \node at (-3.3,0.2) {\scriptsize $\mathbf{e}_1$};
                \node at (-2.5,0.2) {\scriptsize $\mathbf{e}_2$};
                \node at (-2.9,0.85) {\scriptsize $\mathbf{e}_3$};
            \end{tikzpicture}
            \label{fig:est_RI_F1_prior}
        }
        \subfigure[Posterior dist. at $t=0$]{
            \includegraphics[width=0.45\columnwidth, trim=150 80 120 60, clip]{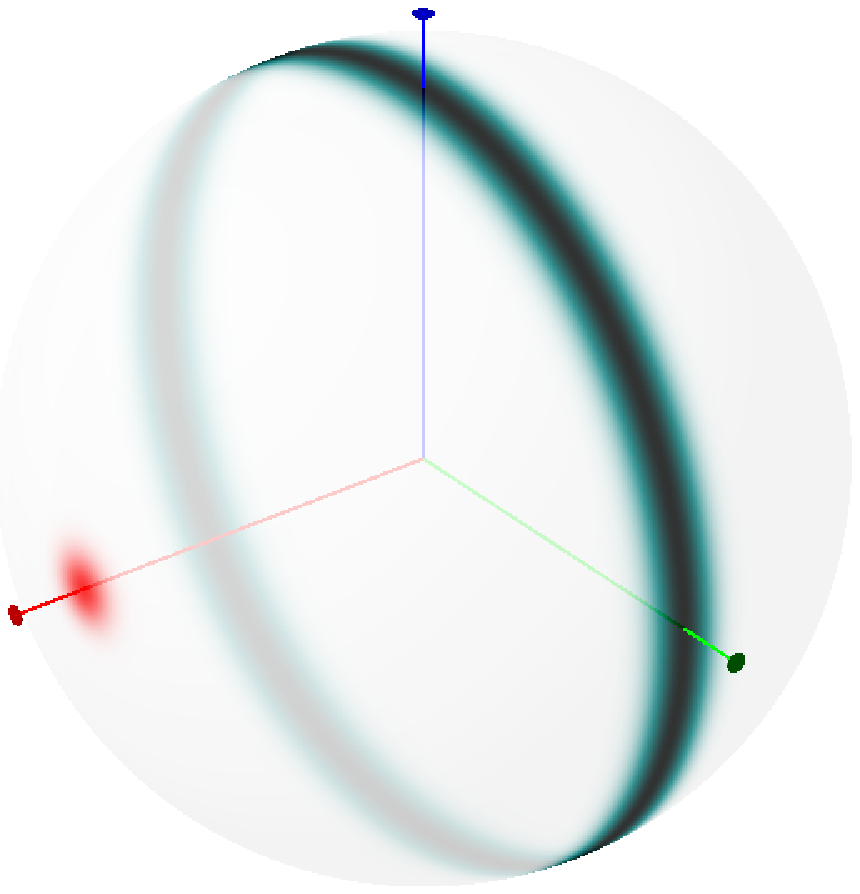}
            \label{fig:est_RI_F1}
        }
    }
    \centerline{
        \subfigure[Propagated dist. at $t=0.5$]{
            \includegraphics[width=0.45\columnwidth, trim=150 80 120 60, clip]{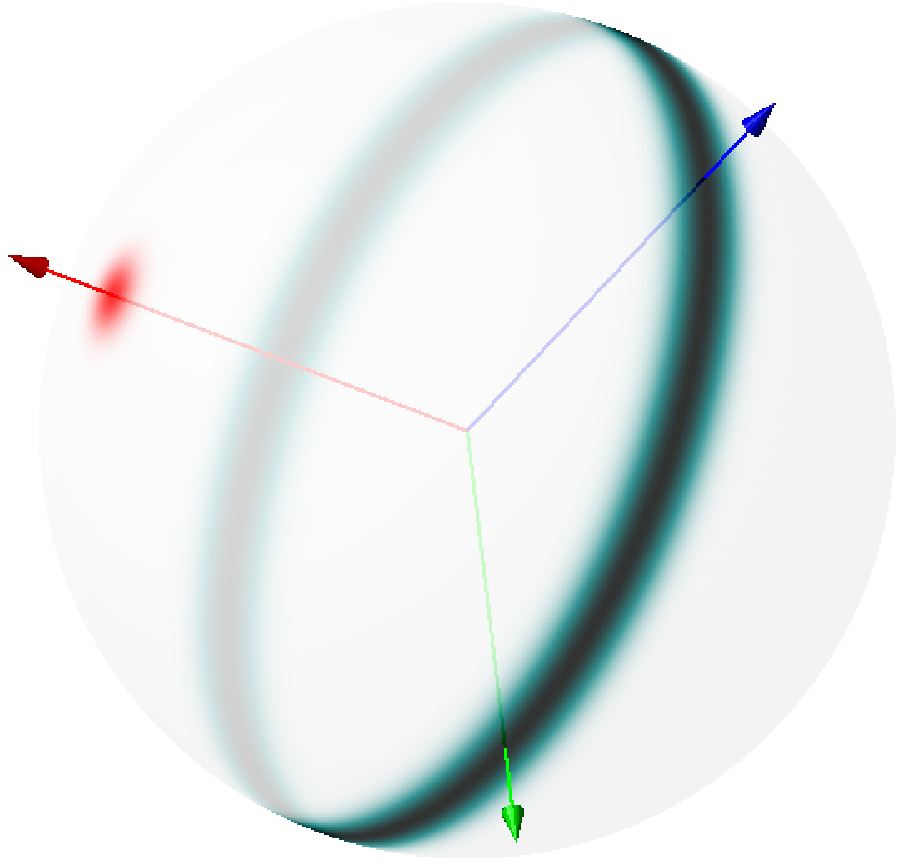}
            \label{fig:est_RI_F3}
        }
        \subfigure[Propagated dist. at $t=1$]{
            \includegraphics[width=0.45\columnwidth, trim=150 80 120 60, clip]{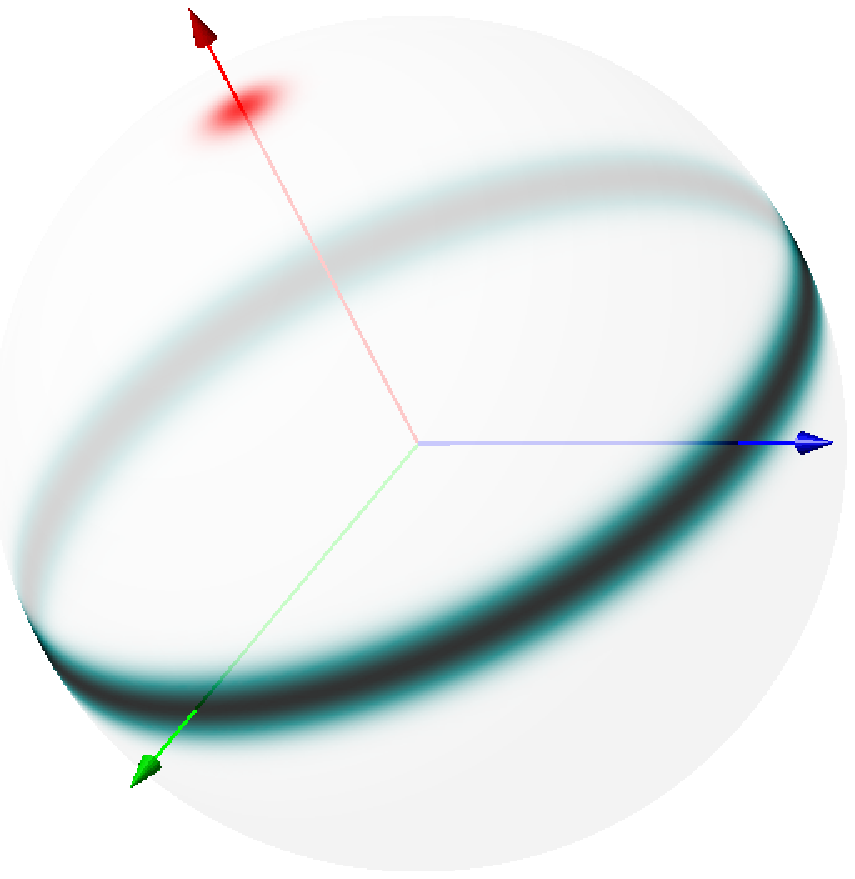}
            \label{fig:est_RI_F5}
        }
    }
    \centerline{
        \subfigure[Propagated dist. overlapped with measured dist. at $t=1$]{
            \includegraphics[width=0.45\columnwidth, trim=150 80 120 60, clip]{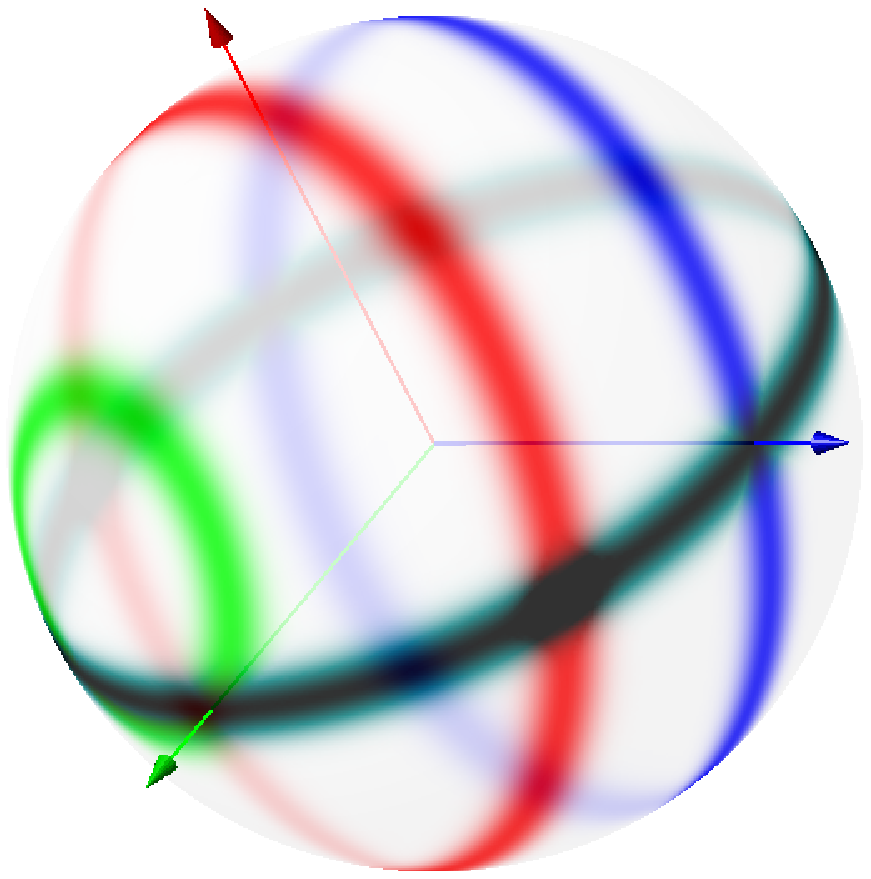}
            \label{fig:est_RI_FN_mea}
        }
        \subfigure[Posterior dist. at $t=1$]{
            \includegraphics[width=0.45\columnwidth, trim=150 80 120 60, clip]{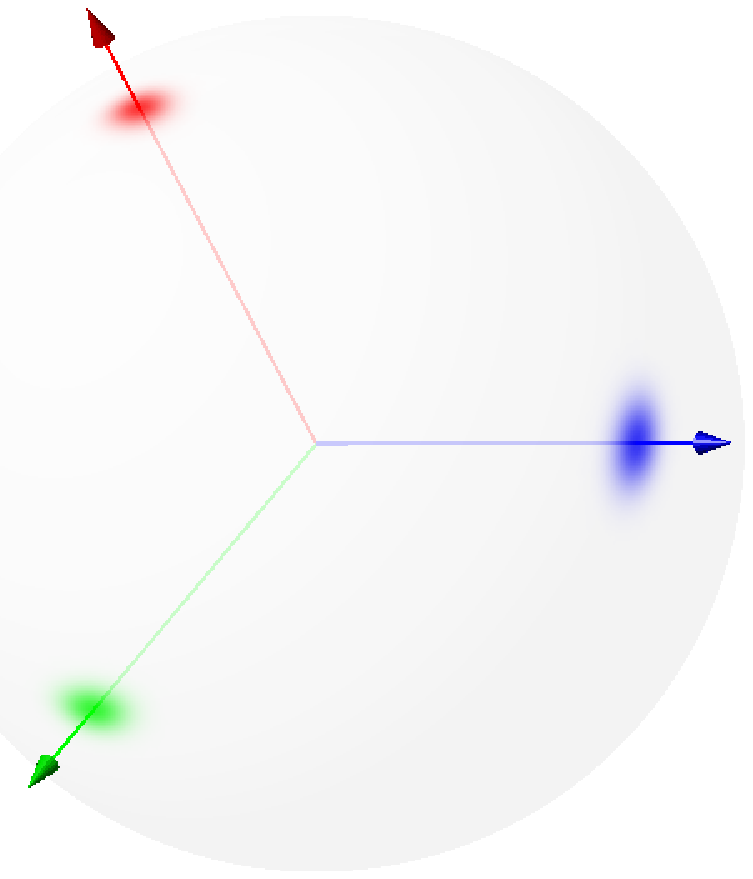}
            \label{fig:est_RI_FN_post}
        }
    }
    \caption{Estimation with the right trivialized \eqref{eqn:dR_w} and the inertial direction measurement \eqref{eqn:F_I}, where the complete attitude is estimated after incorporating two inertial single direction measurements.}\label{fig:est_RI}
\end{figure}

The above intuition is formulated formally in the next theorem.
\begin{theorem} \label{thm:observable}
	Consider the two Bayesian attitude filters composed of
    \begin{itemize}
        \item right-trivialized angular velocity in the inertial frame  \eqref{eqn:F_R_kp} and inertial direction measurement \eqref{eqn:F_I} 
        \item left-trivialized angular velocity in the body-fixed frame \eqref{eqn:F_L_kp}, and body-fixed direction measurement \eqref{eqn:F_B}  
    \end{itemize}
    with the initial distribution $F_0=0_{3\times 3}$.
    Suppose there is some $k_0$ such that $\omega_{k_0} \times a \neq 0$ for the first case, and $\Omega_{k_0} \times b \neq 0$ for the second case. Then the attitude is observable with probability one for both cases.
\end{theorem}
\begin{proof}
    Consider the first case.
	The posterior distribution after the first measurement is given by $F_1 = \kappa ax_1^T$.
	Suppose $k_0=1$ and denote $\exp(h\hat{\omega}_1) = \delta R$, then by \eqref{eqn:SVD_FI+}, $\eqref{eqn:UV_R_kp}$, \eqref{eqn:S_R_kp} and \Cref{lemma:SD}, $U_2^- = \delta R\begin{bmatrix}a&a'&a''\end{bmatrix}$, $S_2^- = \mathrm{diag}([(s_1)_2^-,0,0])$, $V_2^-e_1 = x_1$, where $(s_1)_2^->0$ satisfies
	\begin{align*}
		\frac{1}{c(S_{2}^-)}\frac{\partial c(S_2^-)}{\partial (s_1)_2^-} = (1-h\gamma^2) \frac{1}{c(S_1)}\frac{\partial c(S_1)}{\partial(s_1)_1}.
	\end{align*}
	Therefore the posterior distribution after the second measurement becomes
	\begin{align} \label{eqn:F2}
		F_2 = (s_1)_2^-\delta Rax_1^T + \kappa ax_2^T.
	\end{align}
	Let $\delta Ra = \alpha a + \alpha'a' + \alpha''a''$ for some $\alpha,\alpha',\alpha''\in\Re$. 
    Since $\omega_1\times a\neq 0$, $\alpha'$ and $\alpha''$ cannot both be zeros.
	Then
	\begin{align*}
		F_2 &= (s_1)_2^-(\alpha a + \alpha'a' + \alpha''a'')x_1^T + \kappa ax_2^T \\
		&= a((s_1)_2^-\alpha x_1 + \kappa x_2)^T + (s_1)_2^-\alpha'a'x_1^T + (s_1)_2^-\alpha''a''x_1^T.
	\end{align*}
	Let $(s_1)_2^-\alpha_1x_1 + \kappa x_2 \triangleq v$, and $v',v''$ be arbitrarily chosen such that $\begin{bmatrix}\frac{v}{\norm{v}}&v'&v''\end{bmatrix} \in \SO$.
	Also, let $x_1 = \beta\frac{v}{\norm{v}} + \beta'v' + \beta''v''$. 
    Note that $\beta_2$ and $\beta_3$ cannot both be zeros almost surely.
	Using these, $F_2$ is written as
	\begin{align*}
        F_2 &= \begin{bmatrix}a&a'&a''\end{bmatrix}  \Lambda \begin{bmatrix} \frac{v}{\norm{v}} & v' & v'' \end{bmatrix}^T,
	\end{align*}
    where $\Lambda\in\Re^{3\times 3}$ is
    \begin{align*}
        \Lambda = \begin{bmatrix} \norm{v} & 0 & 0 \\ (s_1)_2^-\alpha'\beta & (s_1)_2^-\alpha'\beta' & (s_1)_2^-\alpha'\beta'' \\ (s_1)_2^-\alpha''\beta & (s_1)_2^-\alpha''\beta' & (s_1)_2^-\alpha''\beta'' \end{bmatrix}.
    \end{align*}
    Since $F_2$ is obtained by multiplying rotation matrices to $\Lambda$, it is starightfoward to see that $F_2$ and $\Lambda$ share the same proper singular values. 
    We have $\det(\Lambda) = 0$, so there is at least one zero singular value.
	However, the rank of $\Lambda$ is two almost surely, as at least one element of the right bottom 2-by-2 block is nonzero.
    Thus, $\Lambda$ has only one zero singular value.
	By the definition of proper singular value decomposition, this concludes $\trs{S_2}I_{3\times 3}-S_2$ is positive-definite, and therefore the attitude is observable with probability one.
	
	Next suppose $k_0>1$. By \eqref{eqn:F2}, $F_2 = a((s_2)_1^-x_1+\kappa x_2)^T$ since $\delta R$ does not rotate $a$, which means both the uncertainty propagation and update steps leave $U_ke_1 = a$, $(s_1)_k>0$, and $(s_2)_k=(s_3)_k=0$.
	Thus the argument in the last paragraph still applies at time $t=t_{k_0}$.
	The proof for the other filter is similar.
\end{proof}

Finally, it should be noted that although we have assumed the attitude follows matrix Fisher distribution which seems an unwanted constraint, it is natural in the sense that if the initial distribution is uniform, and the angular velocity noise in \eqref{eqn:dR_w} and \eqref{eqn:dR_W} is zero, i.e. $H(t)=0$, then the attitude conditioned by single direction measurements follows exactly the matrix Fisher distribution.
Moreover, \Cref{sec:UP} and \Cref{lem:MMSE} together show that the observability of propagated attitude is not affected by the noise.
Therefore, we suppose that the result presented in  \Cref{table:observability} is not specific to the filter; instead, it is inherent to the observed stochastic dynamical system given by \eqref{eqn:dR_w}, \eqref{eqn:dR_W} and \eqref{eqn:x|R}, \eqref{eqn:y|R}.
Indeed, it is shown by simulation and experiment in the next two sections that multiplicative extended Kalman filter also exhibits the same observability. 

\begin{table*}
	\caption{Estimation error (\SI{}{deg}) \label{table:error}}
	\centering
	\begin{tabular}{l|l|c|c|c|c|c|c|c|c}
		\hline
		\multicolumn{2}{l|}{estimator} & \multicolumn{4}{c|}{matrix Fisher} & \multicolumn{4}{c}{MEKF} \\ \hline
		\multicolumn{2}{l|}{case (\textbf{observable}, unobservable)} & \textbf{AVI\_RVI} & AVI\_RVB & AVB\_RVI & \textbf{AVB\_RVB} & \textbf{AVI\_RVI} & AVI\_RVB & AVB\_RVI & \textbf{AVB\_RVB} \\ \hline
		\multirow{2}{*}{simulation} & full error $\pm$ s.d. & 6.57$\pm$0.25 & 90$\pm$0 & 90$\pm$0 & 5.71$\pm$0.15 & 7.02$\pm$0.50 & 135$\pm$24 & 89$\pm$37 & 6.06$\pm$0.27 \\
		& partial error $\pm$ s.d. & - & 3.33$\pm$0.06 & 3.34$\pm$0.06 & - & - & 3.33$\pm$0.06 & 3.33$\pm$0.06 & -  \\ \hline
		\multirow{2}{*}{experiment} & full error & 7.15 & 90.29 & 90.41 & 1.74 & 15.92 & 161.21 & 93.32 & 18.94 \\
		& partial error & - & 0.63 & 3.02 & - & - & 2.95 & 6.43 & -  \\ \hline
	\end{tabular}
\end{table*}

\section{Simulations}

In this section, we show the attitude observability through numerical simulations with the estimator using the matrix Fisher distribution described in Sections III and IV, and the conventional multiplicative extended Kalman filter (MEKF).

\subsection{Simulation settings}

\begin{figure}
	\centering
	\includegraphics{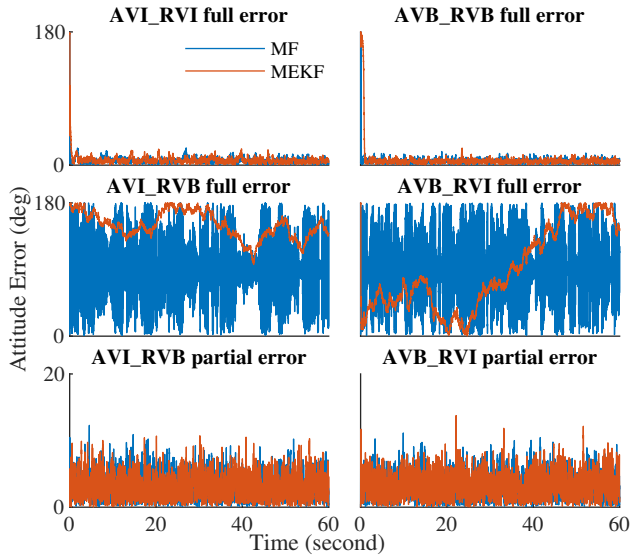}
	\caption{Attitude errors for the matrix Fisher (MF) estimator and MEKF in four combinations of angular velocity and reference vector measurements. \label{fig:attitudeError}}
\end{figure}

We consider a rigid body rotating about its third body-fixed axis at \SI{6}{\radian\per\second}, which simultaneously rotates about the second inertial axis at \SI{1}{\radian\per\second}.
The angular velocity is measured at \SI{150}{\hertz} either in the inertial frame or body-fixed frame, with the isotropic random walk noise of $H = \gamma I_{3\times 3}$ where $\gamma = $ \SI{10}{deg/\sqrt{\second}}.
The reference vector is set as $[1,0,0]$, expressed either in the inertial or body-fixed frame, and measured at \SI{30}{\hertz}.
For the matrix Fisher estimator, the measurement noise is formulated as in \eqref{eqn:x|R} or \eqref{eqn:y|R} with $\kappa=200$, 
and the initial attitude distribution is uniform, i.e., $F_0=0_{3\times 3}$.
The noise parameters and initial conditions for MEKF are chosen similarly as the matrix Fisher estimator.

The full attitude error is defined as the angle between the estimated attitude and true attitude.
For the two unobservable combinations, a partial attitude error which neglects the unobserved degree of freedom is also studied.
More specifically, if the reference vector is known in the inertial frame as $a$, then the partial attitude error is defined as $\arccos(R_t^Ta \cdot R^Ta)$, where $R_t$ is the true attitude;
if the reference vector is known in the body-fixed frame as $b$, the partial attitude error is defined as $\arccos(R_tb \cdot Rb)$.
The four combinations of angular velocity and reference vector measurements are labeled as follows:

\begin{center}
\footnotesize
	\begin{tabular}{l|cc}
		\diagbox[width=10em]{ref. vec.}{ang. vel.} & body-fixed frame & inertial frame \\ \hline
        body-fixed frame & \textbf{AVB\_RVB} & AVI\_RVB \\
        inertial frame & AVB\_RVI & \textbf{AVI\_RVI}
	\end{tabular}
\end{center}
where the boldface font indicates the cases with observability. 

For each case, 100 Monte Carlo simulations with the simulation period of \SI{60}{\second} are carried out.
The attitude error is first averaged across all timestamps in one simulation, and further averaged across simulations.

\subsection{Results}

\begin{figure}
	\centering
	\includegraphics{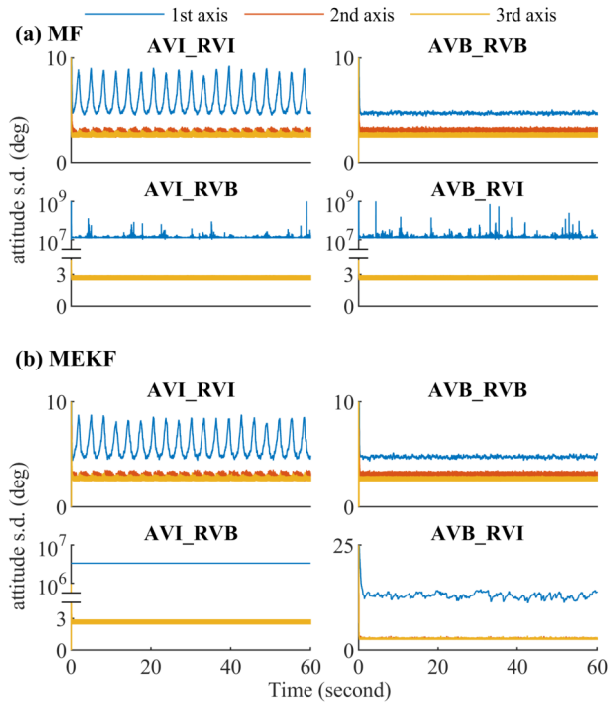}
	\caption{Attitude standard deviations for the matrix Fisher (MF) estimator and MEKF.
		In the ``RVI'' and ``RVB'' cases, the attitude covariance matrix is expressed in the inertial and body-fixed frames respectively.
		For the MF filter, $(\tr{S}I_{3\times 3}-S)^{-1}$ is used as the attitude covariance matrix in the principal axes frame. \label{fig:attitudeStd}}
\end{figure}

The attitude estimation errors are summarized in \Cref{table:error}.
It is clearly shown that for the two observable cases (AVI\_RVI and AVB\_RVB), the full attitude can be estimated with the average error of about \SI{6}{\degree}.
On the other hand, for the two unobservable cases (AVI\_RVB and AVB\_RVI), the full attitude error is around \SI{90}{\degree}, and the partial attitude error is around \SI{3.3}{\degree}.
For a more straightforward comparison, the time evolution of attitude errors for a single simulation is shown in \Cref{fig:attitudeError}, where it is illustrated that for the two observable cases, the full attitude error quickly converges from \SI{180}{\degree} to around zero;
whereas for the two unobservable cases, the full attitude error never converges, but the partial attitude error remains low throughout the simulation since the initial partial error is zero.

The attitude observability is also indicated by the estimated dispersion, quantified as the standard deviation of the error rotation vector as shown in \Cref{fig:attitudeStd}.
For the two observable cases, the standard deviations along all three axes remain below \SI{10}{\degree}, with the rotation about the reference vector (the first axis $e_1$) a little bit larger than the two other axes.
For the two unobservable cases, the standard deviation along the reference vector for the matrix Fisher estimator is mostly above $10^7$ \SI{}{deg}, which  indicates a uniform distribution.
For the MEKF, if the reference vector is known in the body-fixed frame, the standard deviation along the reference vector is similar to the matrix Fisher estimator, which is above $10^6$ \SI{}{deg}.
However, if the reference vector is known in the inertial frame, the standard deviation is around \SI{15}{\degree}, which is only marginally larger than the observable cases.
This discrepancy is caused by that the error rotation vector is expressed in the body-fixed frame, rather than in the inertial frame.
And if the reference vector is known in the inertial frame, MEKF has been shown to apply some slight but erroneous corrections to the rotation about the reference vector due to the linearization of the measurement function~\cite{li2013high}.

\begin{figure}
	\centerline{
		\subfigure[Hardware configuration with reflective markers for a VICON motion capture system]{
			\includegraphics[height=0.19\textwidth]{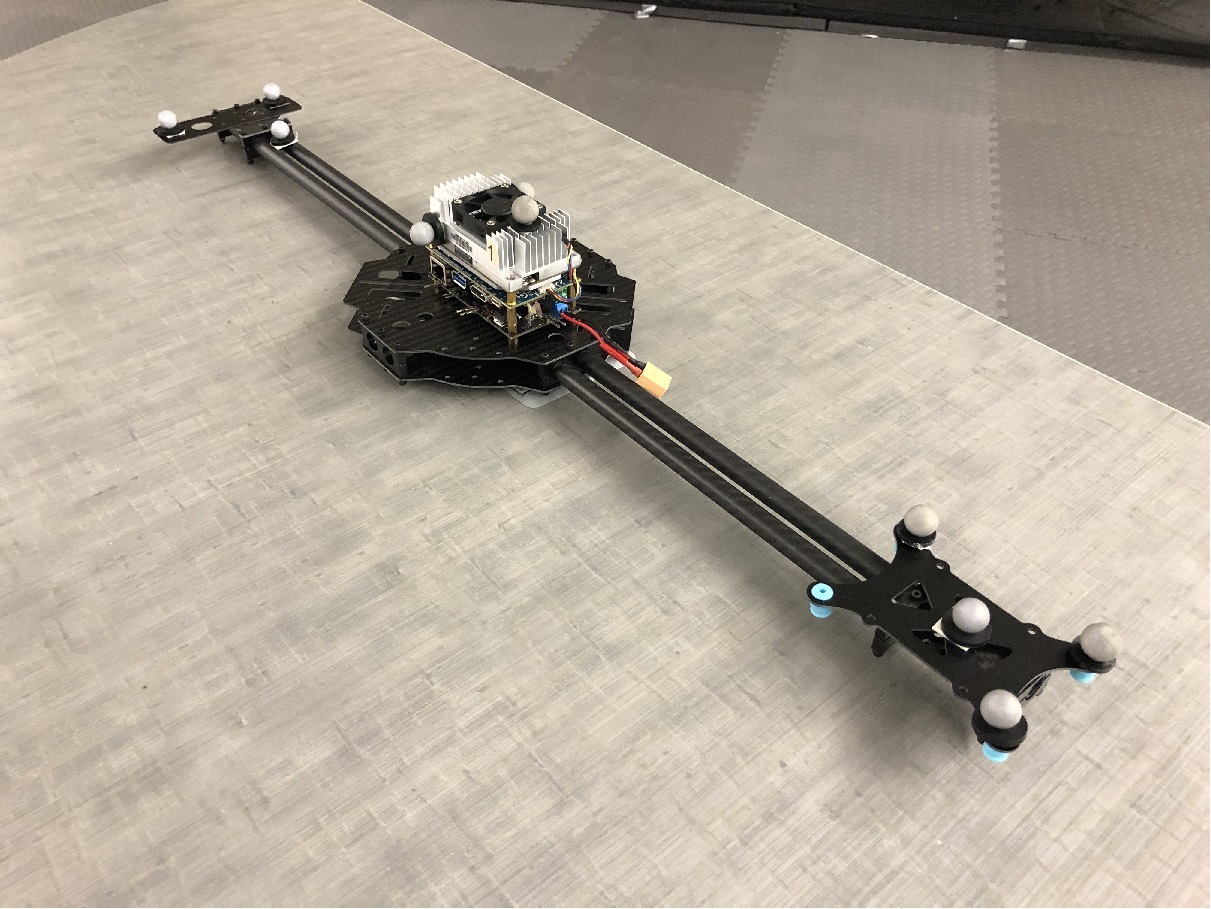}
		}
		\hfill
		\subfigure[Onboard computing module connected to IMU through a custom-made printed circuit board]{
			\includegraphics[height=0.19\textwidth]{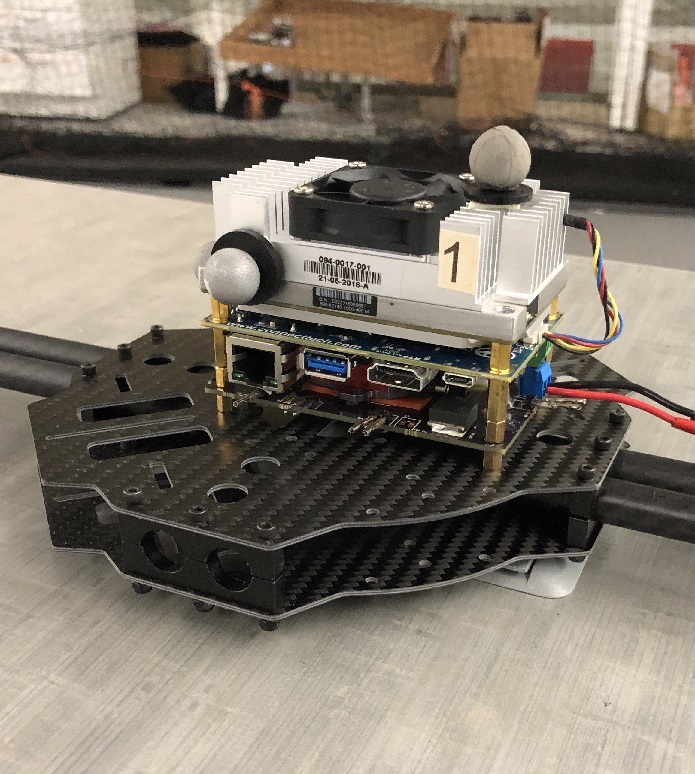}
		}
	}
	\caption{Hardware platform for experiments}\label{fig:exp_hardware}
\end{figure}

\section{Experiments}

The attitude observability was also validated through experiments.
We use a custom-made hardware platform, which has been developed for autonomous unmanned aerial vehicles, to collect measurements while moving it with hands.
An external VICON motion capture system detects reflective markers attached to the platform to determine its attitude, which is used as the ground truth attitude.
A 9-axis inertial measurement unit (VectorNav VN100) is attached to the platform, and the onboard gyroscope provides the angular velociy measurement in the body-fixed frame, which is also converted into the angular velocity measurement in the inertial frame using the ground truth attitude. 
For the inertial direction measurement, the direction of gravity is measured by the accelerometer in IMU.
Moreover, for the body-fixed direction measurement, two additional markers are attached to the vehicle as a known and fixed reference vector in the body-fixed frame, which is measured by the Vicon motion system in the inertial frame.
All Vicon and IMU measurements are synchronized and sampled at \SI{100}{\hertz}, using an onboard Nvidia Jetson TX2 computing module.
The platform was rotated around its roll, pitch and yaw axes during the data collection.
The matrix Fisher estimator and MEKF are run off-board using the collected experimental data, with the single vector measurement update applied at \SI{20}{\hertz}.
The noise parameters and initial conditions are set the same as in the simulation.
Note that these are not carefully tuned for the specific hardware, since the main objective is to verify the observability. 

\begin{figure}
	\centering
	\includegraphics{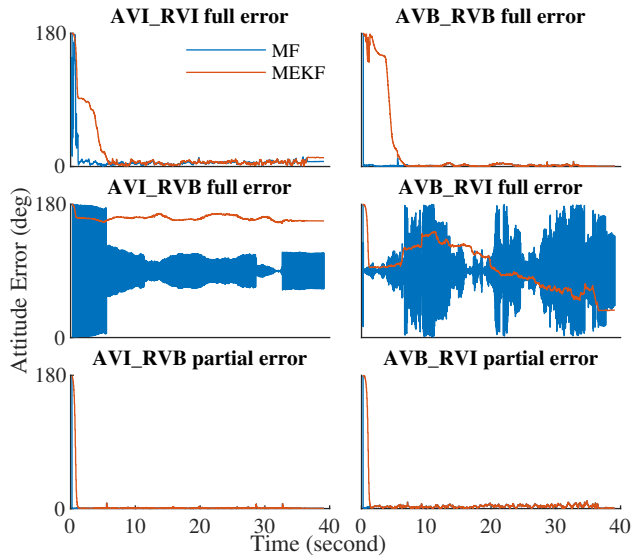}
	\caption{Attitude errors for and matrix Fisher filter and MEKF in experiment. \label{fig:attitudeError-Exp}}
\end{figure}

The attitude errors for the four combinations of measurements are presented in \Cref{table:error} and \Cref{fig:attitudeError-Exp}.
Similar with the simulation results, for the two observable cases (AVI\_RVI and AVB\_RVB) the full attitude error converges to around zero within \SI{10}{\second}; whereas for the two unobservable cases (AVI\_RVB and AVB\_RVI) only the partial error converges. 

\section{Conclusions}

This paper addresses the fundamental question whether the attitude of a rigid body is observable with single direction measurements when combined with stochastic attitude kinematics. 
By observing that the attitude uncertainties are propagated distinctively depending on how the angular velocity measurements are resolved, this paper has discovered two particular cases where the attitude is observable with multiple measurements of a single, fixed reference direction, which has been widely accepted to be impossible. 
This is further studied by formulating stochastic attitude observability through information-theoretic analysis, 
and it is also validated by numerical examples and experimental results. 
For future directions, the effect of a gyro bias on the observability remains to be studied, although preliminary numerical simulations indicate that the presented observability still holds under the bias. 

\appendix

\subsection{Proof of \Cref{thm:ER_R}} \label{app:ER}


Consider a matrix differential equation of $X(t)\in\Re^{n\times n}$,
\begin{align}
    d X(t) = A(t) X(t)  dt,\label{eqn:X_dot}
\end{align}
where $A(t)\in\Re^{n\times n}$ is prescribed, and the initial condition is given by $X(0)=X_0\in\Re^{n\times n}$. 

From the Magnus expansion~\cite{BlaCasPR09}, the solution of \eqref{eqn:X_dot} can be written in terms of matrix exponential as
\begin{align}
    X(t) = \exp C(t) X_0,\label{eqn:X}
\end{align}
where $C(t)$ is the solution of
\begin{align}
    \dot C(t) = \dexp^{-1}_{C(t)} A(t),\quad C(0) = 0.\label{eqn:C_dot}
\end{align}
The operator $\dexp_C^{-1}A$ for $A$ is given by
\begin{align*}
    \dexp^{-1}_{C} A = \sum_{k\geq 0} \frac{B_k}{k!}\ad^k_{C} A,
\end{align*}
where $B_k$ are the Bernoulli numbers, defined by $\sum_{k\geq 0} \frac{B_k}{k!}x^k = \frac{x}{e^x-1}$.
Therefore, \eqref{eqn:C_dot} is expanded into
\begin{align*}
	\dot C= A-\frac{1}{2}[C, A] + \frac{1}{12}[C,[C, A]] + \mathcal{O}(C^4,A),
\end{align*}
Applying the Picard fixed point iteration starting with $C(0)=0$, we have
\begin{align}
    C(t) = C_1(t) + C_2(t) + \mathcal{O}(t^3), \label{eqn:Magnus_B} 
\end{align}
with
\begin{align*}
    C_1(t) & = \int_0^t A(\tau) d\tau,\\
    C_2(t) & = -\frac{1}{2} \int_0^t \bracket{ \int_0^\tau A(\sigma)d\sigma, A(\tau)} d\tau.
\end{align*}
In short, the solution of \eqref{eqn:X_dot} is given by \eqref{eqn:X}, where $C(t)$ is expanded as \eqref{eqn:Magnus_B}.

Next, we find the solution of the stochastic differential equation \eqref{eqn:dR_w}.
We first rewrite it into an equivalent form such that the drift term $\omega dt$ is eliminated, to which the Magnus expansion is applied. 
\begin{prop}
For a given fixed $t$, define $X(\tau)\in\SO$ be
\begin{align}
    X(\tau) =  \exp (-\hat\phi_R (\tau,t))R(\tau), \label{eqn:X_R}
\end{align}
where $\phi_R(\tau, t)\in\Re^3$ is 
\begin{align}
    \hat\phi_R(\tau, t) = \int_{t}^\tau \dexp^{-1}_{\hat\phi_R(\sigma, t)} \hat\omega(\sigma) d\sigma. \label{eqn:phi_R_a}
\end{align}
Then, $X(\tau)=R(t)$ when $\tau =t$ and $X(\tau)$ satisfies the following stochastic differential equation:
\begin{align}
    dX(\tau) =\{ \exp(-\hat\phi_R(\tau, t)) H(\tau) dW \}^\wedge X(\tau ) ,\label{eqn:SDE_XR}
\end{align}
\end{prop}
\begin{proof}
    Since $\phi_R(t,t)=0$, it is straightforward to show $X(t)=R(t)$. 
    Take the differentiation of $X$, we have
    \begin{align*}
        dX(\tau) & =\exp(-\hat\phi_R(\tau,t))  \dexp_{\hat\phi_R(\tau, t)} \deriv{-\hat\phi_R(\tau,t)}{\tau} R(\tau) \nonumber \\
                 & \quad + \exp(-\hat\phi_R(\tau,t)) dR(\tau),
    \end{align*}
    where $\dexp_C$ is the inverse operator of $\dexp_C^{-1}$ and it denotes the derivative of exponential map.
    Thus, the above reduces to
    \begin{align*}
        dX(\tau) & =-\exp(-\hat\phi_R(\tau,t)) \hat\omega(\tau) R(\tau) + \exp(-\hat\phi_R(\tau,t))  dR(\tau),
    \end{align*}
    which becomes \eqref{eqn:SDE_XR} after substituting \eqref{eqn:dR_w}.
\end{proof}

This implies that the stochastic differential equation \eqref{eqn:dR_w} for the attitude kinematics can be transformed into \eqref{eqn:SDE_XR} that has only the diffusion term. 
As such, we can apply the Magnus expansion to \eqref{eqn:SDE_XR} to obtain the solution $X(t)$, which is converted back to the solution $R(t)$ of \eqref{eqn:dR_w}.

\begin{prop}
    An approximate solution of \eqref{eqn:dR_w} is given by
\begin{align}
    R(\tau ) & = \exp(\phi_R(\tau, t))  \{ I_{3\times 3} + \hat q_1(\tau, t) + \hat q_2(\tau, t) + \frac{1}{2}\hat q_1(\tau, t)^2 \} \nonumber \\
             &\quad \times R(t) + \mathcal{O}((\tau -t)^{1.5}), \label{eqn:R_tau_R}
\end{align}
where $\phi_R(\tau, t)$ is defined at \eqref{eqn:phi_R}, and $q_1(\tau,t),q_2(\tau,t)\in\Re^3$ are
\begin{align}
    q_1(\tau, t) & = \int_{t}^\tau \exp(-\hat\phi_R(\sigma,t)) H(\sigma) dW(\sigma),\label{eqn:q1} \\
    q_2(\tau, t) & = -\frac{1}{2} \int_{t}^\tau \int_{t}^\sigma \exp(-\hat\phi_R(\varsigma,t))H(\varsigma)dW(\varsigma) \nonumber \\
                 & \quad \times \exp(-\hat\phi_R(\sigma,t)) H(\sigma) dW(\sigma).\label{eqn:q2}
\end{align}
\end{prop}
\begin{proof}
    As \eqref{eqn:SDE_XR} is formulated in the Stratonovich sense, it is compatible with the usual calculus. 
    Applying the Magnus expansion with $A(t) = \{ \exp(-\hat\phi_R(\tau, t)) H(\tau) dW \}^\wedge$, from \eqref{eqn:X} and \eqref{eqn:Magnus_B},
    \begin{align*}
        X(\tau) = \exp (\hat q_1(\tau, t) +\hat q_2(\tau, t) + \mathcal{O}((\tau -t)^{1.5})) X(t). 
    \end{align*}
    Expanding the exponential map as $\exp Z = I + Z + \frac{1}{2}Z^2 +\cdots$, and applying \eqref{eqn:X_R} with $X(t)=R(t)$ yield  \eqref{eqn:R_tau_R}.
\end{proof}

From the approximate solution \eqref{eqn:R_tau_R}, we can construct the first moment as follows.
By the definition of Wiener process, $\E[dW]=0$ and $E[dW(\tau)dW^T(\sigma)]  = \delta_{\tau,\sigma}I_{3\times 3}  |d\tau| $.
Thus, 
\begin{gather*}
    \E[q_1(\tau, t) ]  = 0,\quad 
    \E[q_2(\tau, t)]  = 0.
\end{gather*}
Next, 
\begin{align}
    \E[q_1(\tau, t) & q_1^T(\tau,t) ] = \int_{t}^\tau\int_{t}^\tau \exp(-\hat\phi_R(\sigma,t)) H(\sigma) \nonumber\\
   &\times \E[ dW(\sigma) dW(\varsigma)^T]  H^T (\varsigma) \exp(\hat\phi_R(\varsigma,t)) \nonumber\\
   & \triangleq G_R(\tau, t). \label{eqn:G_R}
\end{align}
Lastly, using the property of the hat map $\hat x^2 = xx^T - x^T xI_{3\times 3}$ for any $x\in\Re^3$, 
\begin{align*}
    \E[\hat q_1^2] & = \E[q_1q_1^T - q_1^Tq_1 I_{3\times 3}] = G_R(\tau,t) -\trs{G_R(\tau,t)}I_{3\times 3}.
\end{align*}
We obtain \eqref{eqn:ER_tau_R} by taking the mean of \eqref{eqn:R_tau_R} with the above and using the fact that the third order moment of $dW$ is zero.

\subsection{Solution for $\Phi_R(\tau, t)$} \label{app:ER_sim}

\begin{prop}
    Suppose that the angular velocity $\omega(\cdot)$ varies linearly over $[t,\tau]$ from $\omega(t)=\omega_t\in\Re^3$ to $\omega(\tau)=\omega_\tau\in\Re^3$, i.e.,
    \begin{align}
        \omega(\sigma) & = \frac{ (\tau-\sigma)\omega_t + (\sigma-t)\omega_\tau}{\tau-t}.\label{eqn:omega}
    \end{align}
    The corresponding solution $\phi_R(\tau,t)$ of \eqref{eqn:phi_R} is
    \begin{align}
        \phi_R(\tau, t) & =  \frac{1}{2}(\omega_t + \omega_\tau) (t-\tau)	- \frac{1}{12} (\tau-t)^2 \omega_t\times \omega_\tau\nonumber \\
                      & \quad + \mathcal{O}((\tau-t)^3).\label{eqn:phi_omega}
    \end{align}
    Also, the solution $G_R(\tau, t)$ of \eqref{eqn:G} is
    \begin{align}
        G_R(\tau, t) = |\tau-t| H_tH^T_t +\mathcal{O}((\tau-t)^2),\label{eqn:G_omega}
    \end{align}
    where $H_t=H(t)\in\Re^{3\times 3}$.
    These can be substituted into \eqref{eqn:Phi_R} to construct $\Phi_R(\tau,t )$.
    In particular, when $\omega_t=\omega_\tau$, 
    \begin{align*}
        \Phi_R(\tau,t) & = \{ I_{3\times 3} + \frac{1}{2}(\tau-t)(H_t H_t^T - \trs{H_t H_t^T}I_{3\times 3})\}\\
                     & \quad \times\exp(\omega_t(\tau-t)).
    \end{align*}
\end{prop}
\begin{proof}
	Equation \eqref{eqn:phi_omega} can be obtained by substituting \eqref{eqn:omega} into \eqref{eqn:Magnus_B} and completing the integration.
	Since $\E[R(\tau)]$ is a second order approximation, the same order approximation to $G(t,\tau)$ is
	\begin{align*}
	G_R(\tau,t) & = |\tau -t| \exp(\hat\phi_R(t,t)) H_t H_t^T \exp(-\hat\phi_R(t,t))\\
	& \quad + \mathcal{O}((t-\tau)^2).
	\end{align*}
	As $\phi(t,t)=0$, this reduces to \eqref{eqn:G_omega}.
\end{proof}

\subsection{Deterministic Attitude Observability} \label{app:deterministic}

In this appendix, attitude observability with single direction measurements in deterministic sense is studied, and the results are consistent with Theorem \ref{thm:non-estimable} and Theorem \ref{thm:observable}.
The proof is based on Theorem 3.1 and Theorem 3.12 in \cite{hermann1977nonlinear}, and notations therein.

\begin{theorem}
	Let the deterministic inertial and body-fixed direction measurements be
	\begin{align}
		x &= R^Ta, \label{eqn:x(R)} \\
		y &= Rb, \label{eqn:y(R)}
	\end{align}
	respectively, where $a,b\in\Sph^2$ are reference vectors.
	Then the system \eqref{eqn:R_dot_w} and \eqref{eqn:x(R)} is weakly locally observable, the system \eqref{eqn:R_dot_w} and \eqref{eqn:y(R)} is unobservable.
\end{theorem}
\begin{proof}
	Without loss of generality, we assume $a = b = \begin{bmatrix} 1 & 0 & 0 \end{bmatrix}^T$.
	By definition, both systems are locally controllable.
	For any $R\in\SO$ and $\hat{\eta}\in\so$, the Lie derivative of $x(R)$ along $\hat{\eta}R$ is
	\begin{align}
		(L_{(\hat{\eta}R)}x)(R) = \frac{d}{dt} \Big\lvert_{t=0} R^T\exp(t\hat{\eta})^Ta = R^T\hat{\eta}^Ta.
	\end{align}
	Note that $\{L_{(\hat{e}_1R)}x, L_{(\hat{e}_2R)}x, L_{(\hat{e}_3R)}x\} \subset \mathcal{G}$.
	For any $R_0\in\SO$, define a local coordinate as $R(\theta) = \exp(\theta_1\hat{e}_1+\theta_2\hat{e}_2+\theta_3\hat{e}_3)R_0$, then $dR = \widehat{d\theta}R_0$, and
	\begin{align*}
		d(L_{\hat{e}_2R_0}x)(R_0) &= R_0^T \begin{bmatrix} -d\theta_2 & d\theta_1 & 0 \end{bmatrix}^T, \\
		d(L_{\hat{e}_3R_0}x)(R_0) &= R_0^T \begin{bmatrix} -d\theta_3 & 0 & d\theta_1 \end{bmatrix}^T.
	\end{align*}
	This indicates the dimension of $d\mathcal{G}(R_0)$ is three, and therefore the system \eqref{eqn:R_dot_w} and \eqref{eqn:x(R)} is weakly locally observable.
	
	On the other hand, for any $R\in\SO$ and $\hat{\eta}_1, \hat{\eta}_2\in\so$, the Lie derivative of $y(R)$ is
	\begin{align}
		\big(L_{(\hat{\eta}_1R)}(L_{(\hat{\eta}_2R)}y)\big)(R) = \hat{\eta}_1\hat{\eta}_2Rb.
	\end{align}
	And
	\begin{align*}
		d\big(L_{(\hat{\eta}_1R)}(L_{(\hat{\eta}_2R)}y)\big)(R_0) = \hat{\eta}_1\hat{\eta}_2\widehat{d\theta}(R_0)_1 = -\hat{\eta}_1\hat{\eta}_2 \widehat{(R_0)_1} d\theta,
	\end{align*}
	where $(R_0)_1$ is the first column of $R_0$.
	Then the dimension of $d\mathcal{G}(R_0)$ is two since the rank of $\widehat{(R_0)_1}$ is two.
	Because the system \eqref{eqn:R_dot_w} and \eqref{eqn:y(R)} is locally controllable, it is unobservable.
\end{proof}




\vspace{-1cm}
\begin{IEEEbiographynophoto}{Weixin Wang}
	received his M.S. degree at University of Wisconsin-Madison, WI, USA, in 2018, and the B.E. degree at Tsinghua University, Beijing, China, in 2016, both in mechanical engineering.
	He is currently pursuing his Ph.D. degree at George Washington University.
	His research interests include nonlinear estimation theory, inertial sensors, and human movement tracking.
\end{IEEEbiographynophoto}

\vspace{-1cm}
\begin{IEEEbiographynophoto}{Kanishke Gamagedara}
	is currently pursuing his Ph.D. degree at the George Washington University, where he also received his master's degree in Mechanical and Aerospace Engineering in 2018.
	He received his bachelor's degree at University of Peradeniya, Sri Lanka, in Mechanical Engineering in 2015.
	His research interests are mainly focused on geometric control of unmanned aerial systems.
\end{IEEEbiographynophoto}

\vspace{-1cm}
\begin{IEEEbiographynophoto}{Taeyoung Lee}
	is an associate professor of the Department of Mechanical and Aerospace Engineering at the George Washington University. He received his doctoral degree in Aerospace Engineering and his master's degree in Mathematics at the University of Michigan in 2008. His research interests include geometric mechanics and control with applications to complex aerospace systems. 
\end{IEEEbiographynophoto}

\end{document}